\documentclass[twoside,12pt]{article} 
\usepackage[margin=1.0in]{geometry}

\usepackage{amssymb,amsmath,amsthm}
\usepackage{mathtools}
\usepackage{bm,lscape}
\usepackage{bbm}
\usepackage{mathrsfs,dsfont}
\usepackage{wasysym}
\RequirePackage[
    colorlinks=true,
    linkcolor=blue,
    urlcolor=blue,
    citecolor=blue]{hyperref}
\RequirePackage{natbib}
\usepackage[utf8]{inputenc}
\usepackage[english]{babel}
\usepackage{algorithm}
\usepackage{algpseudocode}

\usepackage{graphicx}
\usepackage{color}
\usepackage{rotating}
\usepackage{authblk}
\usepackage{appendix}

\allowdisplaybreaks

\newcommand{\dR}{\mathbb{R}}

\newcommand{\dP}{\mathbb{P}}
\newcommand{\cB}{\mathcal{B}}
\newcommand{\cF}{\mathcal{F}}
\newcommand{\cL}{\mathcal{L}}

\newcommand{\dE}{\mathbb{E}}
\newcommand{\rI}{\mathrm{I}}

\allowdisplaybreaks

\newcommand{\E}{\mathds{E}}

\renewcommand{\P}{\mathds{P}}

\newcommand{\ddr}{\mathrm{d}}

\newtheorem{thm}{Theorem}[section]

\newtheorem{remark}[thm]{Remark}

\providecommand{\keywords}[1]
{
  \small	
  \textbf{\textit{Keywords:}} #1
}

\begin{document}

\title{A martingale approach to Gaussian fluctuations and laws of iterated logarithm for Ewens-Pitman model}


\author[1]{Bernard Bercu\thanks{bernard.bercu@math.u-bordeaux.fr}}
\author[2]{Stefano Favaro\thanks{stefano.favaro@unito.it}}
\affil[1]{\small{Institut de Math\'ematiques de Bordeaux, Universit\'e de Bordeaux, France}}
\affil[2]{\small{Department of Economics and Statistics, University of Torino and Collegio Carlo Alberto, Italy}}

\maketitle

\begin{abstract}
The Ewens-Pitman model refers to a distribution for random partitions of $[n]=\{1,\ldots,n\}$, which is indexed by a pair of parameters 
$\alpha \in [0,1)$ and $\theta>-\alpha$, with $\alpha=0$ corresponding to the Ewens model in population genetics. The large $n$ asymptotic properties of the Ewens-Pitman model have been the subject of numerous studies, with the focus being on the number $K_{n}$ of partition sets and the number $K_{r,n}$ of partition subsets of size $r$, for $r=1,\ldots,n$. While for $\alpha=0$ asymptotic results have been obtained in terms of almost-sure convergence and Gaussian fluctuations, for $\alpha\in(0,1)$ only almost-sure convergences are available, with the proof for $K_{r,n}$ being given only as a sketch. In this paper, we make use of martingales to develop a unified and comprehensive treatment of the large $n$ asymptotic behaviours of $K_{n}$ and $K_{r,n}$ for $\alpha\in(0,1)$, providing alternative, and rigorous, proofs of the almost-sure convergences of $K_{n}$ and $K_{r,n}$, and covering the gap of Gaussian fluctuations. We also obtain new laws of the iterated logarithm for $K_{n}$ and $K_{r,n}$.
\end{abstract}



\section{Introduction}

The Ewens-Pitman (EP) model refers to a distribution for random partitions, introduced by \citet{Pit(95)} as a generalization of the celebrated Ewens model in population genetics \citep{Ewe(72)}. For $n\in\mathbb{N}$, consider a random partition of $[n]=\{1,\ldots,n\}$ into $K_{n}\in\{1,\ldots,n\}$ partition subsets $\{A_1, \ldots,A_{K_n}\}$
of corresponding sizes $\mathbf{N}_{n}=(N_{1,n},\ldots,N_{K_{n},n})$, where $N_{k,n}$ is the cardinal number of set $A_k$, for $k=1,\ldots, K_{n}$, such that 
\begin{equation*}
n=\sum_{k=1}^{K_{n}}N_{k,n}.
\end{equation*}
For any real numbers $\alpha$ in $[0,1)$ and $\theta>-\alpha$, the EP model assigns to the random vector $(K_{n},\mathbf{N}_{n})$ the probability 
\begin{equation}\label{epsm}
\P(K_{n}=k,\mathbf{N}_{n}=(n_{1},\ldots,n_{k}))=\frac{n!}{k!}\frac{\left(\frac{\theta}{\alpha}\right)^{(k)}}{(\theta)^{(n)}}\prod_{i=1}^{k}\frac{\alpha(1-\alpha)^{(n_{i}-1)}}{n_{i}!}
\end{equation}
where, for any $a\in \dR$, $(a)^{(n)}$ stands for the rising factorial of $a$ of order $n$, that is $(a)^{(n)}=a(a+1)\cdots(a+n-1)$. Denote by $K_{r,n}$ the number of partition subsets of size $r$, given for all $r=1,\ldots,n$, by
\begin{equation*}
K_{r,n}=\sum_{k=1}^{K_{n}}\rI_{\{N_{k,n}=r\}}
\end{equation*}
where $\rI$ is the indicator function. One can easily see that
\begin{equation*}
n=\sum_{r=1}^{n} rK_{r,n}
\hspace{1cm}\text{and}\hspace{1cm}
K_{n}=\sum_{r=1}^{n}K_{r,n}.
\end{equation*}
The distribution of $\mathbf{K}_{n}=(K_{1,n},\ldots,K_{n,n})$ is known as EP sampling formula or frequency-of-frequencies distribution, and it follows by means of a combinatorial rearrangement of \eqref{epsm}. The Ewens model is recovered from \eqref{epsm} by setting $\alpha=0$. We refer the reader to \citet{Pit(95),Pit(06)} for more details. The joint distribution \eqref{epsm} admits several constructions, with the most common being a sequential or generative construction through the Chinese restaurant process \citep{Pit(95),Fen(98),Zab(05)} and a Poisson process construction through random sampling from the Pitman-Yor random probability measure \citep{Per(92),Pit(97)}; see also \citet{Dol(21)} for a construction through a class of negative Binomial compound Poisson sampling models \citep{Cha(07)}. The EP model plays a critical role in a variety of research fields, such as population genetics, Bayesian nonparametric statistics, excursion theory, combinatorics, machine learning and statistical physics. We refer to the monograph by \citet[Chapter 3 and Chapter 4]{Pit(06)} for a comprehensive account on the EP model and generalizations thereof.

\vspace{1ex}
Under the EP model, there have been several studies on the large $n$ asymptotic behaviour of $K_{n}$ and the $K_{r,n}$'s; see \citet[Chapter 3]{Pit(06)}. In particular, for $\alpha=0$, \citet[Theorem 2.3]{Kor(73)} exploited the fact that $K_{n}$ is a sum of $n$ independent Bernoulli random variables, and showed that
\begin{equation}\label{as_e}
\lim_{n\rightarrow+\infty}\frac{K_{n}}{\log n}=\theta\qquad\text{a.s.}
\end{equation}
Furthermore, it follows directly from Lindeberg-L\'evy central limit theorem that
\begin{equation}\label{gf_e}
\frac{K_{n}-\theta\log n}{\sqrt{\log n}} 
\underset{n\rightarrow+\infty}{\overset{\cL}{\longrightarrow}}
\sqrt{\theta}\,\mathcal{N}(0,1),
\end{equation}
with $\mathcal{N}(0,1)$ being a standard Gaussian random variable. See \citet{Fen(98)} and references therein for some functional versions of both \eqref{as_e} and \eqref{gf_e}. As regard to the $K_{r,n}$'s, from \citet[Theorem 1]{Arr(92)} it follows that 
\begin{equation}\label{w_e_1}
K_{r,n}
\underset{n\rightarrow+\infty}{\overset{\cL}{\longrightarrow}}
Z_{r}
\end{equation}
and
\begin{equation}\label{w_e_2}
(K_{1,n},K_{2,n},\ldots)
\underset{n\rightarrow+\infty}{\overset{\cL}{\longrightarrow}}
(Z_{1},Z_{2},\ldots),
\end{equation}
where the $Z_{r}$'s are independent Poisson random variables with $\E[Z_{r}]=\theta/r$, for all $r\geq1$. We refer to \citet{Arr(03)} for various generalizations and refinements of these asymptotic results, and their interplay with combinatorics.

For any $\alpha\in(0,1)$, $K_{n}$ is no longer a sum of independent Bernoulli random variables. However, \citet[Theorem 3.8]{Pit(06)} made use of a martingale construction for $K_{n}$, based on a likelihood ratio defined through the distribution \eqref{epsm}, and applied the strong law of large numbers for martingales to prove that
\begin{equation}\label{as_ep_1}
\lim_{n\rightarrow+\infty}\frac{K_{n}}{n^{\alpha}}=S_{\alpha,\theta}\qquad\text{a.s.},
\end{equation}
where $S_{\alpha,\theta}$ is a positive and almost surely finite random variable. If $f_{\alpha}$ is the positive $\alpha$-Stable density function, then the distribution of $S_{\alpha,\theta}$ has density function 
\begin{equation}\label{pitman_div}
f_{S_{\alpha,\theta}}(s)=\frac{\Gamma(\theta+1)}{\alpha\Gamma(\theta/\alpha+1)}s^{\frac{\theta-1}{\alpha}-1}f_{\alpha}(s^{-1/\alpha}),
\end{equation}
where $\Gamma$ stands for the Euler Gamma function. Precisely, \eqref{pitman_div} is a generalization of the Mittag-Leffler density function, which is recovered by setting $\theta=0$ \citep{Zol(86)}. As regard to the $K_{r,n}$'s, from \citet[Lemma 3.11]{Pit(06)} it follows that
\begin{equation}\label{as_ep_2}
\lim_{n\rightarrow+\infty}\frac{K_{r,n}}{n^{\alpha}}=p_{\alpha}(r)S_{\alpha,\theta}\qquad\text{a.s.},
\end{equation}
where
\begin{displaymath}
p_{\alpha}(r)=\frac{\alpha(1-\alpha)^{(r-1)}}{r!}.
\end{displaymath}
The almost-sure convergences \eqref{as_ep_1} and \eqref{as_ep_2} are the natural counterparts of \eqref{as_e} and \eqref{w_e_1}, respectively, though at different scales. Instead, it is still an open problem of obtaining a counterpart of \eqref{gf_e}, as well as Gaussian fluctuations for $K_{r,n}$.

\subsection{Our contributions}

For $\alpha\in(0,1)$ and $\theta>-\alpha$, we make use of martingales to develop a unified and comprehensive treatment of the large $n$ asymptotic behaviours of $K_{n}$ and $K_{r,n}$, providing alternative, and rigorous, proofs of the almost-sure limits \eqref{as_ep_1} and \eqref{as_ep_2}, and covering the gap of Gaussian fluctuations. In particular, we propose a new martingale construction for $K_{n}$, simpler than that applied in \citet[Theorem 3.8]{Pit(06)}, which leads to an alternative proof of \eqref{as_ep_1}, still relying on the strong law of large numbers for martingales. Further, by exploiting our martingale construction, as well as \eqref{as_ep_1}, we show that 
\begin{equation}\label{cl_k}
\sqrt{n^\alpha} \left(\frac{K_n}{n^{\alpha}}- S_{\alpha,\theta} \right)
\underset{n\rightarrow+\infty}{\overset{\cL}{\longrightarrow}}
\sqrt{S_{\alpha,\theta}^\prime}\, \mathcal{N}(0,1),
\end{equation}
where $S^{\prime}_{\alpha,\theta}$ is random variable independent of $\mathcal{N}(0,1)$ and sharing the same distribution as $S_{\alpha,\theta}$. We also prove the law of the iterated logarithm for $K_{n}$, i.e.,
\begin{equation}
\label{lil_k}
\limsup_{n \rightarrow \infty} \frac{\big(K_n-n^{\alpha}S_{\alpha,\theta}\big)^2}{2n^\alpha \log \log n} 
=  S_{\alpha,\theta}\qquad\text{a.s.}
\end{equation}
Then, we extend our analysis to the $K_{r,n}$'s, for which only a sketch of the proof of \eqref{as_ep_2} is given in \citet[Lemma 3.11]{Pit(06)}. In particular, we introduce a martingale construction for $K_{r,n}$, providing a rigorous proof of \eqref{as_ep_2} by means of the strong law of large numbers for martingales. Further, by exploiting our martingale construction and \eqref{as_ep_2}, we establish a Gaussian fluctuation as well as the law of iterated logarithm for $K_{r,n}$. Critical for our study is the work of \citet{Hey(77)}, which provides sufficient conditions to achieve Gaussian fluctuations and laws of iterated logarithm from the martingale convergence theorem.

\subsection{Organization of the paper}

The paper is structured as follows. In Section \ref{sec1}, we present the martingale construction for $K_{n}$ and the proofs of the almost-sure convergence, as well as the $\mathbb{L}^{p}$ convergence, the Gaussian fluctuation, and the law of iterated logarithm. Section \ref{sec2} contains an analogous asymptotic analysis for $K_{r,n}$. In Section \ref{sec3}, we discuss some directions of future research. The appendices are devoted to the sequential construction of the EP model and an alternative proof, without relying on martingales, of the $\mathbb{L}^{2}$ convergence of $K_{n}$ and $K_{r,n}$.


\section{Asymptotic results for the number of partition sets}\label{sec1}
We start by introducing the keystone martingale construction for $K_{n}$. This is a critical tool at the basis of all the results in this section. From the sequential construction of the EP model \citep[Proposition 9]{Pit(95)}, for all $n \geq 1$,
\begin{equation}\label{sequential_k}
   \dP(K_{n+1}  = K_{n}+k\,|\,K_n) =\left \{ \begin{array}{ccc}
    {\displaystyle \frac{\alpha K_n+\theta}{n+\theta} }  & \text{ if } & k=1, \vspace{1ex}\\[0.4cm]
     {\displaystyle \frac{n -  \alpha K_n}{n+\theta} } & \text{ if }  & k=0.
   \end{array} \right.
\end{equation}
We refer the reader to Appendix A for more detDensity deconvolution addresses the estimation of the unknown (probability) density function $f$ of a random signal from data that are observed with an independent additive random noise. This is a classical problem in statistics, for which frequentist and Bayesian nonparametric approaches are available to deal with static or batch data. In this paper, we consider the problem of density deconvolution in a streaming or online setting where noisy data arrive progressively, with no predetermined sample size, and we develop a sequential nonparametric approach to estimate $f$. By relying on a quasi-Bayesian sequential approach, often referred to as Newton's algorithm, we obtain estimates of $f$ that are of easy evaluation, computationally efficient, and with a computational cost that remains constant as the amount of data increases, which is critical in the streaming setting. Large sample asymptotic properties of the proposed estimates are studied, yielding provable guarantees with respect to the estimation of $f$ at a point (local) and on an interval (uniform). In particular, we establish local and uniform central limit theorems, providing corresponding asymptotic credible intervals and bands. We validate empirically our methods on synthetic and real data, by considering the common setting of Laplace and Gaussian noise distributions, and make a comparison with respect to the kernel-based approach and a Bayesian nonparametric approach with a Dirichlet process mixture prior.ails on \eqref{sequential_k}. Equation \eqref{sequential_k} simply means that 
\begin{equation}\label{DECKN}
K_{n+1} = K_n + \xi_{n+1}
\end{equation}
where the conditional distribution of the random variable $\xi_{n+1}$, given the $\sigma$-algebra $\cF_n=\sigma(K_1, \ldots,K_n)$, is 
the Bernoulli $\cB(p_n)$ distribution with parameter
\begin{equation}\label{DEFPN}
p_n=\frac{\alpha K_n+\theta}{n+\theta}.
\end{equation} 
According to the above definition of the sequence $(\xi_{n})$, as $\dE[\xi_{n+1}\, |\, \cF_n]=p_n$, we have almost surely
\begin{equation}\label{CONDM1KN}
\dE[K_{n+1}\,|\,\cF_{n}]=\dE[K_{n}+\xi_{n+1}\,|\,\cF_{n}]=K_n+p_n=
\beta_{n}K_{n}+\frac{\theta}{n+\theta}
\end{equation}
where 
\begin{equation}
\label{DEFBETAN}
\beta_{n}=1+\frac{\alpha}{n+\theta}.
\end{equation}
Consequently, we obtain from \eqref{CONDM1KN} that almost surely
\begin{equation}
\label{CONDM1KNMN}
\dE\left[K_{n+1} + \frac{\theta}{\alpha}\,\Big \vert\,\cF_{n}\right]= \beta_n K_n + \frac{\theta}{n+\theta}+ \frac{\theta}{\alpha}= \beta_n \left(K_n+\frac{\theta}{\alpha}\right).
\end{equation}
Let $(b_{n})$ be the sequence defined by $b_{1}=1$ and for all $n\geq2$,
\begin{equation}\label{DEFBN}
b_{n}\!=\!\prod_{k=1}^{n-1}\beta_{k}^{-1}\!=\!\prod_{k=1}^{n-1}\left(\frac{k+\theta}{k+\alpha+\theta}\right)\!=\!
\left \{ \!\! \begin{array}{ccc}
    {\displaystyle \!\! \Big(\frac{\alpha+\theta}{\theta}\Big)\frac{(\theta)^{(n)}}{(\alpha+\theta)^{(n)}}}  & \text{ if } & \theta \neq 0, \vspace{1ex}\\[0.4cm]
     {\displaystyle \frac{\alpha (n-1)!}{(\alpha)^{(n)}} }& \text{ if }  & \theta=0.
   \end{array}  \right.
\end{equation}
Denote by $(M_n)$ the sequence of random variables defined, for all $n\geq 1$, by
\begin{equation}\label{DEFMN}
M_{n}=b_{n}K_{n}+\left(\frac{\alpha+\theta}{\alpha}\right)\frac{(\theta)^{(n)}}{(\alpha+\theta)^{(n)}}
=b_{n} \left(K_{n}+\frac{\theta}{\alpha}\right).
\end{equation}
Since $b_{n}=\beta_{n}b_{n+1}$, \eqref{CONDM1KNMN} immediately implies that for all $n\geq1$,
\begin{equation*}
\dE[M_{n+1}\,|\,\cF_{n}]=b_{n+1}\beta_n\left(K_{n}+\frac{\theta}{\alpha} \right)=b_{n}\left( K_{n}+\frac{\theta}{\alpha} \right)=M_{n}
\end{equation*}
almost surely. Moreover, $(M_n)$ is square integrable as $K_n \leq n$. Consequently, $(M_{n})$ is a locally square integrable martingale.
Moreover, noting that
$$
M_{n+1}-M_{n}=b_{n+1}\big(\xi_{n+1} - \dE[\xi_{n+1}\,|\,\cF_{n}]\big),
$$
we have almost surely
\begin{equation*}
\dE\big[(M_{n+1}\!-\!M_{n})^2\,|\,\cF_{n}\big]\!=\!b_{n+1}^2 \big(\dE\big[\xi_{n+1}^2\,|\,\cF_{n}\big]\!-\!\dE^2[\xi_{n+1}\,|\,\cF_{n}]\big) \!=\! b_{n+1}^2 p_n(1-p_n),
\end{equation*}
which leads, via \eqref{DEFPN}, to
\begin{equation}\label{COND2MN}
\dE\big[(M_{n+1}-M_{n})^2\,|\,\cF_{n}\big]=b_{n+1}^2 \left(\frac{(\theta + \alpha K_{n})(n-\alpha K_n)}{(n+\theta)^2}\right)\qquad \text{a.s.}
\end{equation}

\subsection{An alternative approach for the almost-sure convergence}

Based on the above martingale construction, we present an alternative proof of \eqref{as_e}, by relying on the strong law of large numbers for martingale. Our proof is more natural and intuitive than that of \citet[Theorem 3.8]{Pit(06)},  which makes use of the exchangeable partition probability function induced by the EP model.

\begin{thm}
\label{T-ASCVG}
Let $K_{n}$ be the number of partition sets in a random partition of $[n]$ distributed according to the EP model with $\alpha\in(0,1)$ and $\theta>-\alpha$. Then,
\begin{equation}\label{ASCVG}
\lim_{n \rightarrow+\infty}
\frac{K_{n}}{n^{\alpha}}=S_{\alpha,\theta}\qquad\text{a.s.}
\end{equation}
where $S_{\alpha,\theta}$ is a positive and almost surely finite random variable whose distribution has probability density function \eqref{pitman_div}. This convergence holds in $\mathbb{L}^p$ for any 
integer $p\geq 1$, 
\begin{equation}\label{ASCVGLP}
 \lim_{n \rightarrow+\infty} \dE\left[ \left| \frac{K_n}{n^{\alpha}} -S_{\alpha,\theta} \right|^p \right]=0.
\end{equation}
\end{thm}

\begin{proof}
We already saw that $(M_{n})$ is a locally square integrable martingale.
Denote by $\langle M \rangle_n$ its predictable quadratic variation,
\begin{equation*}
\langle M \rangle_n  = \sum_{k=0}^{n-1} \dE[(M_{k+1}-M_k)^2\,|\,\cF_{k}].
\end{equation*}
We have from \eqref{COND2MN} that
\begin{equation}\label{PQVMN}
\langle M \rangle_n  =\left(\frac{\alpha+\theta}{\alpha}\right)^2+\sum_{k=1}^{n-1} b_{k+1}^2 \left(\frac{(\theta + \alpha K_{k})(k-\alpha K_k)}{(k+\theta)^2}\right)\qquad\text{a.s.}
\end{equation}
Taking the expectation on both sides of \eqref{PQVMN}, we obtain that
\begin{equation*}
\dE\big[\langle M \rangle_n\big]  =\left(\frac{\alpha+\theta}{\alpha}\right)^2+\sum_{k=1}^{n-1} b_{k+1}^2 
\left(\frac{\theta k + \alpha(k-\theta) \dE[K_{k}] -\alpha^2 \dE[K_k^2]}{(k+\theta)^2}\right).
\end{equation*}
Hence, it exists some constant $C>0$ such that for $n$ large enough,
\begin{equation}\label{BDPQVMN}
\dE\big[\langle M \rangle_n\big]  \leq C \sum_{k=1}^{n} \frac{b_{k}^2 \dE[K_{k}]}{k}.
\end{equation}
Moreover, we have from \eqref{CONDM1KN} and \eqref{DEFBN} that
if $\theta \neq 0$,
\begin{equation}\label{MOM1KN}
\dE[K_n]=b_n^{-1}\sum_{k=0}^{n-1} \frac{(\theta)^{(k)}}{(\alpha+\theta+1)^{(k)}}
=\frac{\theta}{\alpha}\left( \frac{(\alpha+ \theta)^{(n)} - (\theta)^{(n)}}{(\theta)^{(n)}}\right),
\end{equation}
whereas if $\theta=0$,
\begin{equation}\label{MOM1KN0}
\dE[K_n]=\prod_{k=1}^{n-1} \left(1 +\frac{\alpha}{k} \right)
= \frac{(\alpha)^{(n)}}{\alpha (n-1)!}.
\end{equation}
In addition, it follows from standard results on the asymptotic behavior of the 
Euler Gamma function that for $\theta \neq 0$,
\begin{equation}\label{CVGBNNN}
\lim_{n\rightarrow+\infty} \frac{1}{n^\alpha}
\frac{(\alpha+ \theta)^{(n)}}{(\theta)^{(n)}}=
\lim_{n\rightarrow+\infty} \frac{\Gamma(\theta +1) \Gamma(n+\alpha+\theta)}{n^\alpha 
\theta \Gamma(\alpha + \theta) \Gamma(n+\theta)}
=\frac{\Gamma(\theta+1)}{\theta \Gamma(\alpha + \theta)}.
\end{equation}
Hence, we obtain from \eqref{MOM1KN}, \eqref{MOM1KN0} and \eqref{CVGBNNN} that 
whatever the value of $\theta>-\alpha$,
\begin{equation*}
\lim_{n\rightarrow+\infty}\frac{1}{n^\alpha}\dE[K_{n}]= 
\frac{\Gamma(\theta+1)}{\alpha \Gamma(\alpha+ \theta)}.
\end{equation*}
Consequently, we deduce from \eqref{BDPQVMN}, \eqref{DEFBN} and \eqref{CVGBNNN} that
\begin{equation}\label{SUPPQVMN}
\sup_{n\geq 1} \dE\big[\langle M \rangle_n\big] <+\infty.
\end{equation}
Therefore, we deduce from \eqref{SUPPQVMN} that $(M_n)$ is bounded in $\mathbb{L}^2$.
Then, it follows from Doob's martingale convergence theorem \citep[Corollary 2.2]{Hal(80)}, that the sequence $(M_n)$ 
converges a.s. and in $\mathbb{L}^2$ to a square integrable random variable $M_{\alpha,\theta}$.
In fact, we shall prove below that $(M_n)$ is bounded in $\mathbb{L}^p$ for any integer $p \geq 1$
which means that $(M_n)$ converges a.s. and in $\mathbb{L}^p$ to $M_{\alpha,\theta}$
for any integer $p \geq 1$. Hereafter, we obtain from \eqref{DEFMN} and \eqref{CVGBNNN} that
\begin{equation*}
\lim_{n \rightarrow+\infty}
\frac{K_{n}}{n^{\alpha}}=S_{\alpha,\theta}\qquad\text{a.s.},
\end{equation*}
where the random variables $M_{\alpha,\theta}$ and $S_{\alpha,\theta}$ are tightly related by means of the identity, which follows from \eqref{CVGBN} below,
$$M_{\alpha,\theta}=\Big( \frac{\Gamma(\alpha+\theta+1)}{\Gamma(\theta+1)}\Big)S_{\alpha,\theta},$$
completing the proof of \eqref{ASCVG}. We shall now proceed to the proof of the convergence
in $\mathbb{L}^p$ for any integer $p \geq 1$. In particular, we are going to compute the 
falling factorial moments $\dE[(K_n)_{(p)}]$ of $K_n$ where, for any $a \in \dR$, $(a)_{(p)}=a(a-1)\cdots(a-p+1)$ 
with $(a)_{(0)}=1$. By an application of Vandermonde identity for the falling factorial together with \eqref{DECKN}, for any integer $p \geq 1$,
$$
(K_{n+1})_{(p)}=(K_n+\xi_{n+1})_{(p)}= \sum_{k=0}^{p}{p\choose k}(K_{n})_{(k)}(\xi_{n+1})_{(p-k)}.
$$
By taking the conditional expectation on both sides of this identity, we obtain that
\begin{equation}\label{CONDMKNP}
\dE[(K_{n+1})_{(p)}\,|\,\cF_{n}]=\sum_{k=0}^{p}{p\choose k}(K_{n})_{(k)}\dE[(\xi_{n+1})_{(p-k)}\,|\,\cF_{n}]\qquad\text{a.s.},
\end{equation}
where $\dE[(\xi_{n+1})_{(1)}\,|\,\cF_{n}]=\dE[\xi_{n+1}\,|\,\cF_{n}]=p_n$, and $\dE[(\xi_{n+1})_{(k)}\,|\,\cF_{n}]=0$ for all $k \geq 2$. Accordingly, we can deduce from \eqref{DEFPN} and \eqref{CONDMKNP} that for any integer $p \geq 1$,
\begin{align}
\dE[(K_{n+1})_{(p)}\,|\,\cF_{n}]&=(K_{n})_{(p)}+p(K_{n})_{(p-1)}\Big(\frac{\alpha K_n+\theta}{n+\theta}\Big)\qquad\text{a.s.}\notag \\
&=(K_{n})_{(p)}+\frac{\alpha p}{n+\theta}(K_{n})_{(p-1)}(K_n -p+1)+L_n(p)\qquad\text{a.s.} \notag \\
&= \Big(1+\frac{\alpha p}{n+\theta}\Big) (K_{n})_{(p)} +L_n(p)\qquad\text{a.s.}
\label{MKNP1}
\end{align}
where 
$$
L_n(p)=\Big(\frac{p( \alpha(p-1)+\theta)}{n+\theta}\Big)(K_{n})_{(p-1)}.
$$
Denote $b_1(p)=1$ and for all $n \geq 2$,
\begin{equation}
\label{DEFBNP}
b_n(p)=\prod_{k=1}^{n-1}\left(\frac{k+\alpha p+\theta}{k+\theta}\right)
=
\left \{ \begin{array}{ccc}
    {\displaystyle \Big(\frac{\theta}{\alpha p+ \theta}\Big)\frac{(\alpha p + \theta)^{(n)}}{(\theta)^{(n)}}} & \text{ if } & \theta \neq 0, \vspace{1ex}\\[0.4cm]
     {\displaystyle \frac{(\alpha p)^{(n)}}{\alpha p (n-1)!} }& \text{ if }  & \theta=0.
   \end{array}  \right.
\end{equation}
It follows from \eqref{MKNP1} and \eqref{DEFBNP} that for all $n \geq 2$ and for any integer $p \geq 2$,
\begin{displaymath}
\dE[(K_{n})_{(p)}]=b_n(p)\sum_{k=1}^{n-1}\big(b_{k+1}(p)\big)^{-1}\dE[L_k(p)]
\end{displaymath}
leading, for $\theta \neq 0$, to
\begin{equation}\label{MKNP2}
\dE[(K_{n})_{(p)}]\!=\!
\Big(\frac{p( \alpha(p-1)+\theta)}{\alpha p +\theta}\Big)
\Big(\frac{(\alpha p + \theta)^{(n)}}{(\theta)^{(n)}}\Big)
\sum_{k=1}^{n-1}\Big(\frac{(\theta)^{(k)}\dE[(K_{k})_{(p-1)}]}{(\alpha p +\theta+1)^{(k)}}\Big)
\end{equation}
while, for $\theta = 0$, to
\begin{equation}\label{MKNP20}
\dE[(K_{n})_{(p)}]=
\Big(\frac{(p-1) (\alpha p)^{(n)}}{(n-1)!} \Big)
\sum_{k=1}^{n-1}\Big(\frac{(k-1)!\dE[(K_{k})_{(p-1)}]}{(\alpha p +1)^{(k)}}\Big).
\end{equation}
We deduce from \eqref{MOM1KN} and \eqref{MKNP2} that for $\theta \neq 0$,
\begin{displaymath}
\dE[(K_{n})_{(1)}]=\frac{\theta}{\alpha (\theta)^{(n)} }\big((\alpha+ \theta)^{(n)} - (\theta)^{(n)}\big),
\end{displaymath}
\begin{displaymath}
\dE[(K_{n})_{(2)}]=\frac{\theta(\theta + \alpha)}{\alpha^2 (\theta)^{(n)}}\big((2\alpha+ \theta)^{(n)}-2(\alpha+ \theta)^{(n)} + (\theta)^{(n)}\big),
\end{displaymath}
and more generally, for all $n \geq 1$ and for any integer $p \geq 1$,
\begin{align}
\dE[(K_{n})_{(p)}]&=\frac{\prod_{k=0}^{p-1} (k\alpha +\theta)}{\alpha^p (\theta)^{(n)} }
\sum_{k=0}^p (-1)^{p-k} {p\choose k} (k \alpha + \theta)^{(n)}, \notag \\
&=\frac{\big(\frac{\theta}{\alpha} \big)^{(p)}}{(\theta)^{(n)}} \sum_{k=0}^p (-1)^{p-k} {p\choose k} (k \alpha + \theta)^{(n)}.
\label{MKNP3}
\end{align}
In addition, we have from \eqref{MOM1KN0} and \eqref{MKNP20} that
for $\theta = 0$,
\begin{displaymath}
\dE[(K_{n})_{(1)}]=\frac{(\alpha)^{(n)}}{\alpha (n-1)!},
\end{displaymath}
\begin{displaymath}
\dE[(K_{n})_{(2)}]=\frac{1}{\alpha (n-1)!}\big((2\alpha)^{(n)}-2(\alpha)^{(n)} \big),
\end{displaymath}
and more generally, for all $n \geq 1$ and for any integer $p \geq 1$,
\begin{equation}\label{MKNP30}
\dE[(K_{n})_{(p)}]
=\frac{(p-1)!}{\alpha (n-1)!}
\sum_{k=1}^p (-1)^{p-k} {p\choose k} (k \alpha )^{(n)}.
\end{equation}
We obtain from \eqref{DEFBNP} together with \eqref{MKNP3} and \eqref{MKNP30} that it exists a
positive constant $C(p,\alpha, \theta)$ such that for all $n \geq 1$ and $p \geq 1$,
\begin{equation}
\label{MKNP4}
\dE[(K_{n})_{(p)}] \leq C(p,\alpha, \theta) b_n(p).
\end{equation}
Moreover, we also have the elementary inequality
\begin{equation}
\label{MKNP5}
\dE[K_{n}^p] \leq p^p + p! \dE[(K_{n})_{(p)}\rI_{K_n \geq p}].
\end{equation}
Consequently, we find from \eqref{DEFMN}, \eqref{MKNP4} and \eqref{MKNP5}
that it exists a positive constant $D(p,\alpha, \theta)$ such that for all $n \geq 1$ and $p \geq 1$,
\begin{equation}
\label{MKNP6}
\dE[M_n^p] \leq D(p,\alpha, \theta)  b_n(p) b_n^p.
\end{equation}
Furthermore, it is easy to see from \eqref{DEFBN} and \eqref{DEFBNP} that
\begin{equation}
\label{MKNP7}
b_n(p) b_n^p=\prod_{k=1}^{n-1}\left(1+\frac{\alpha p}{k+\theta}\right)
\left(1+\frac{\alpha}{k+\theta}\right)^{\!\!-p} \leq 1.
\end{equation}
Hence, we deduce from \eqref{MKNP6} and \eqref{MKNP7} that
\begin{equation*}
\sup_{n \geq 1}\dE[M_n^p] <+\infty,
\end{equation*}
which means that the martingale $(M_n)$ is bounded in $\mathbb{L}^p$ for any $p \geq 1$. Therefore, it follows from Doob's martingale convergence theorem \citep[Corollary 2.2]{Hal(80)}, the sequence $(M_n)$ converges almost surely and in 
$\mathbb{L}^p$ to a finite random variable $M_{\alpha,\theta}$. Our goal is now to compute all the moments of $M_{\alpha,\theta}$. We shall only carry out the proof for $\theta \neq 0$ inasmuch as the proof for $\theta=0$ follows exactly the same arguments.
First of all, as in \eqref{CVGBNNN}, 
\begin{equation*}
\lim_{n\rightarrow+\infty}\frac{(\alpha p + \theta)^{(n)}}{n^{\alpha p }(\theta)^{(n)}}=
\frac{\Gamma(\theta +1)}{\theta \Gamma(\alpha p +\theta)}.
\end{equation*}
Hence, we obtain once again from \eqref{MKNP3} that for any integer $p \geq 1$,
\begin{equation}
\label{LIMMKN1}
\lim_{n \rightarrow+\infty} \frac{1}{n^{\alpha p}}\dE[(K_{n})_{(p)}]=
\frac{\Gamma(\theta +1)}{\theta \Gamma(\alpha p +\theta)}\Big(\frac{\theta}{\alpha}\Big)^{\!(p)\!}.
\end{equation}
However, we have for all $p\geq 1$,
\begin{equation}
\label{LIMMKN2}
\dE[K_{n}^p]=\sum_{k=0}^p \left\{ 
\begin{matrix}
p \\k 
\end{matrix}
\right\}
\dE[(K_{n})_{(k)}],
\end{equation}
where the curly brackets are the Stirling numbers of the second kind given by
$$
\left\{ 
\begin{matrix}
p \\k 
\end{matrix}
\right\}= \frac{1}{k!}\sum_{i=0}^k (-1)^{k-i}  {k\choose i} i^p.
$$
Since, 
$\Big\{ 
\begin{matrix}
p \\p 
\end{matrix}
\Big\}=1
$, we obtain from \eqref{LIMMKN1} and \eqref{LIMMKN2} that for all $p \geq 1$,
\begin{equation}
\label{CVGMEANKNP}
\lim_{n \rightarrow+\infty} \frac{1}{n^{\alpha p}}\dE[K_{n}^p]=\frac{\Gamma(\theta +1)}{\theta \Gamma(\alpha p +\theta)}\Big(\frac{\theta}{\alpha}\Big)^{\!(p)\!}.
\end{equation}
We also recall from \eqref{DEFBN} that
\begin{equation}
\label{CVGBN}
\lim_{n\rightarrow+\infty} n^{\alpha}b_{n}=\lim_{n\rightarrow+\infty}\frac{n^{\alpha}(\theta+1)_{(n)}}{(\alpha+\theta+1)_{(n)}}=\frac{\Gamma(\alpha+\theta+1)}{\Gamma(\theta+1)}.
\end{equation}
Consequently, it follows from the conjunction of \eqref{DEFMN}, \eqref{CVGMEANKNP} and \eqref{CVGBN} that
\begin{equation*}
\lim_{n\rightarrow+\infty} \dE[M_{n}^{p}]=\lim_{n\rightarrow+\infty} 
b_{n}^{\,p}\dE\Big[\Big(K_{n}+\frac{\theta}{\alpha}\Big)^{p}\Big]
=\lim_{n\rightarrow+\infty} b_{n}^{\,p} \sum_{k=0}^p {p\choose k} \dE[K_{n}^k]
\left( \frac{\theta}{\alpha}\right)^{p-k},
\end{equation*}
which drastically reduces to
\begin{equation}
\label{CVGMEANMNP}
\lim_{n\rightarrow+\infty} \dE[M_{n}^{p}]=
\lim_{n\rightarrow+\infty} b_{n}^{\,p}  \dE[K_{n}^p]
\!=\! \left(\frac{\Gamma(\alpha+\theta+1)}{\Gamma(\theta+1)}\right)^{p}\!\!
\frac{\Gamma(\theta+1)}{\theta \Gamma(\alpha p +\theta)}\!\left(\frac{\theta}{\alpha}\right)^{(p)}
\!\!\!\!\!\!\!.
\end{equation}
It immediately leads, for all $p \geq 1$, to
$$
\dE[M_{\alpha,\theta}^p]=\left(\frac{\Gamma(\alpha+\theta+1)}{\Gamma(\theta+1)}\right)^{p}
\frac{\Gamma(\theta+1)}{\theta \Gamma(\alpha p +\theta)}\left(\frac{\theta}{\alpha}\right)^{(p)},
$$
as well as
$$
\dE[S_{\alpha,\theta}^p]=
\frac{\Gamma(\theta+1)}{\theta \Gamma(\alpha p +\theta)}\left(\frac{\theta}{\alpha}\right)^{(p)}.
$$
The distribution of the random variable $S_{\alpha,\theta}$ has 
probability density function given by \eqref{pitman_div}.
Finally, \eqref{ASCVGLP}
follows from \eqref{ASCVG} and \eqref{CVGMEANKNP} together with Riesz-Scheff\'e theorem,
which completes the proof of Theorem \ref{T-ASCVG}.
\end{proof}

\begin{remark}
An alternative proof of the convergence in $\mathbb{L}^{2}$ given by \eqref{ASCVGLP}, which follows directly from the Poisson process construction of the EP model \citep[Proposition 9]{Pit(03)} can be found in Appendix B. See also \citet[Chapter 4]{Pit(06)}.
\end{remark}

\subsection{The Gaussian fluctuation}

Our martingale approach for $K_{n}$ suggests that, in addition to the almost-sure convergence, it is possible to establish a Gaussian fluctuations of $K_n$, properly normalized, around its almost-sure limit $S_{\alpha,\theta}$.

\begin{thm}
\label{T-FLUCBLOCK}
Let $K_{n}$ be the number of partition sets in a random partition of $[n]$ distributed according to the EP model with $\alpha\in(0,1)$ and $\theta>-\alpha$. Then,
\begin{equation}
\label{ANBLOCK1}
\frac{K_n-n^{\alpha}S_{\alpha,\theta}}{\sqrt{K_n}}
\underset{n\rightarrow+\infty}{\overset{\cL}{\longrightarrow}}
\mathcal{N}(0,1),
\end{equation}
where $\mathcal{N}(0,1)$ denotes a standard Gaussian random variable. Moreover, we also
have
\begin{equation}
\label{ANBLOCK2}
\sqrt{n^\alpha} \left(\frac{K_n}{n^{\alpha}}- S_{\alpha,\theta} \right)
\underset{n\rightarrow+\infty}{\overset{\cL}{\longrightarrow}}
\sqrt{S_{\alpha,\theta}^\prime}\, \mathcal{N}(0,1),
\end{equation}
where $S_{\alpha,\theta}^\prime$ is a random variable independent of $\mathcal{N}(0,1)$ and sharing the same distribution as $S_{\alpha,\theta}$.
\end{thm}

\begin{proof}
The proof relies on a beautiful result due to \cite{Hey(77)} concerning with Gaussian fluctuations for martingales.
For all $n\geq 1$, denote $\Delta M_{n+1}=M_{n+1}-M_{n}$ and
$$
\Lambda_{n}=\sum_{k=n}^{\infty}\dE\big[\Delta M_{k+1}^2|\cF_{k}\big].
$$
We immediately have from \eqref{COND2MN} that 
\begin{equation}
\label{FLUC1}
\Lambda_{n}
=\sum_{k=n}^{\infty}b_{k+1}^2 \left(\frac{(\theta + \alpha K_{k})(k-\alpha K_k)}{(k+\theta)^2}\right).
\end{equation}
On the one hand, we know from \eqref{ASCVG} that
\begin{equation}
\label{FLUC2}
\lim_{n \rightarrow+\infty} \frac{K_n}{n^\alpha}=S_{\alpha,\theta}\qquad\text{a.s.}
\end{equation}
On the other hand, we already saw from \eqref{CVGBN} that
\begin{equation}
\label{FLUC3}
\lim_{n \rightarrow+\infty} n^{\alpha}b_{n}=\frac{\Gamma(\alpha+\theta+1)}{\Gamma(\theta+1)}.
\end{equation}
Consequently, we deduce from \eqref{FLUC1}, \eqref{FLUC2} and \eqref{FLUC3} that
\begin{equation}\label{FLUC4}
\lim_{n \rightarrow+\infty} n^{\alpha}\Lambda_n= \Big( \frac{\Gamma(\alpha+\theta+1)}{\Gamma(\theta+1)}\Big)^2 S_{\alpha,\theta}\qquad\text{a.s.}
\end{equation}
By the same token, if
$$
s_{n}^2=\dE[\Lambda_{n}]=\sum_{k=n}^{\infty}\dE\big[\Delta M_{k+1}^2\big],
$$
we also have
\begin{equation*}
s_{n}^2
=\sum_{k=n}^{\infty}b_{k+1}^2 \left(\frac{\theta k + \alpha(k-\theta) \dE[K_{k}] -\alpha^2 \dE[K_k^2]}{(k+\theta)^2}\right).
\end{equation*}
However, we already saw from \eqref{CVGMEANKNP} that
\begin{equation}
\lim_{n\rightarrow+\infty}\frac{1}{n^\alpha}\dE[K_{n}]= 
\frac{\Gamma(\theta+1)}{\alpha \Gamma(\alpha+\theta)}
\label{LIMKN1}
\end{equation}
and
\begin{equation}
\label{LIMKN2}
\lim_{n\rightarrow+\infty}\frac{1}{n^{2\alpha}}\dE[K_{n}^2]=
\frac{(\alpha+\theta) \Gamma(\theta+1)}{\alpha^2 \Gamma(2\alpha+\theta)}.
\end{equation}
It ensures that
\begin{displaymath}
\lim_{n \rightarrow+\infty} n^{\alpha}s_n^2= \Big( \frac{ \alpha+\theta }{\alpha}\Big)
\frac{\Gamma(\alpha+\theta+1)}{\Gamma(\theta+1)}.
\end{displaymath} 
Hereafter, we claim that for any $\eta >0$,
\begin{equation}\label{FLUC6}
\lim_{n \rightarrow+\infty} n^{\alpha}\sum_{k=n}^{\infty}\dE\big[\Delta M_{k+1}^2 \rI_{\{|\Delta M_{k+1}|>\eta \sqrt{n^{-\alpha}} \}}\big]=0.
\end{equation}
As a matter of fact, we clearly have  for any $\eta >0$,
$$
n^{\alpha}\sum_{k=n}^{\infty}\dE\big[\Delta M_{k+1}^2 \rI_{\{|\Delta M_{k+1}|>\eta \sqrt{n^{-\alpha}} \}}\big]
\leq 
\frac{n^{2\alpha}}{\eta^2}\sum_{k=n}^{\infty}\dE\big[\Delta M_{k+1}^4\big].
$$
Moreover, it follows from tedious but straightforward calculations that for all $n \geq 1$,
\begin{equation}
\label{DECMN4COND}
\dE\big[\Delta M_{n+1}^4 | \cF_n\big]=b_{n+1}^4 \sum_{p=0}^4 e_n(p)K_n^p
\end{equation}
where
\begin{align*}
e_n(0)&=\frac{n \theta}{(n+\theta)^4}(n^2-n\theta +\theta^2), \\
e_n(1)&=\frac{\alpha( n-\theta) }{(n+\theta)^4}(n^2-4n\theta +\theta^2), \\
e_n(2)&=\frac{-2\alpha^2( 2n-\theta)(n-2\theta)}{(n+\theta)^4}, \\
e_n(3)&= \frac{6\alpha^3(n-\theta)}{(n+\theta)^4}, \\
e_n(4)&=- \frac{3\alpha^4}{n^4}.
\end{align*}
Taking the expectation on both sides of \eqref{DECMN4COND}, we obtain that for all $n \geq 1$,
\begin{equation}
\label{DECMN4}
\dE\big[\Delta M_{n+1}^4 \big]=b_{n+1}^4 \sum_{p=0}^4 e_n(p)\dE[K_n^p]
\end{equation}
Furthermore, we already saw from \eqref{CVGMEANKNP} that
\begin{equation}
\label{LIMKN3}
\lim_{n\rightarrow+\infty}\frac{1}{n^{3\alpha}}\dE[K_{n}^3]=
\frac{(\alpha+\theta)(2\alpha+\theta) \Gamma(\theta+1)}{\alpha^3 \Gamma(3\alpha+\theta)}
\end{equation}
and
\begin{equation}
\label{LIMKN4}
\lim_{n\rightarrow+\infty}\frac{1}{n^{4\alpha}}\dE[K_{n}^4]=
\frac{(\alpha+\theta)(2\alpha+\theta)(3\alpha+\theta) 
\Gamma(\theta+1)}{\alpha^4 \Gamma(4\alpha+\theta)}.
\end{equation}
Hence, we deduce from \eqref{FLUC3} and \eqref{DECMN4} together with \eqref{LIMKN1}, \eqref{LIMKN2}, \eqref{LIMKN3} and \eqref{LIMKN4} that
\begin{equation}
\label{FLUC7}
\lim_{n \rightarrow+\infty} n^{1+3\alpha}\dE\big[\Delta M_{n+1}^4 \big]= (\alpha + \theta)
\Big( \frac{\Gamma(\alpha+\theta+1)}{\Gamma(\theta+1)}\Big)^3.
\end{equation} 
Consequently, we obtain from \eqref{FLUC7} that
$$
\lim_{n \rightarrow+\infty} n^{3\alpha} \sum_{k=n}^{\infty}\dE\big[\Delta M_{k+1}^4\big]
= (\alpha + \theta)
\Big( \frac{\Gamma(\alpha+\theta+1)}{\Gamma(\theta+1)}\Big)^3
$$
which clearly leads to \eqref{FLUC6}. Finally, let $(P_n)$ be the martingale defined by
\begin{equation}
\label{DEFMARTPN}
P_{n+1}=\sum_{k=1}^n k^{\alpha} \big(\Delta M_{k+1}^2 - \dE[\Delta M_{k+1}^2\,|\, \cF_{k}]\big).
\end{equation}
Its predictable quadratic variation is given by
$$
\langle P \rangle_{n+1} = \sum_{k=1}^n k^{2\alpha} 
\big(\dE[\Delta M_{k+1}^4\,|\, \cF_{k}]- (\dE[\Delta M_{k+1}^2\,|\,\cF_{k}])^2\big).
$$
Therefore, we find from \eqref{DECMN4COND} that
\vspace{-1ex}
\begin{equation*}
\langle P \rangle_{n+1} \leq \sum_{p=0}^4\sum_{k=1}^n k^{2 \alpha}b_{k+1}^4e_{k}(p)K_{k}^p.
\end{equation*}
One can observe that we always have $e_n(4) < 0$. For $n$ large enough, $e_n(2) <0$.
Furthermore, it exists some constant $C>0$ such that for $n$ large enough, $e_n(0) \leq C/n$,
$e_n(1) \leq C/n$ and $e_n(3) \leq C/n^3$. Consequently, we get that for $n$ large enough,
\begin{equation}
\label{CVGMARTCN}
\langle P \rangle_{n} \leq C \left( \sum_{k=1}^n \frac{k^{2 \alpha}b_{k}^4}{k}
+\sum_{k=1}^n \frac{k^{2 \alpha}b_{k}^4K_k}{k} + \sum_{k=1}^n \frac{k^{2 \alpha}b_{k}^4K_k^3}{k^3}
\right).
\end{equation}
Hence, we obtain from \eqref{FLUC2}, \eqref{FLUC3} and \eqref{CVGMARTCN} that $\langle P \rangle_n$ converges 
a.s. to a finite random variable. Then, we deduce from the strong law of large numbers 
for martingales given by the first part of  \cite[Theorem 1.3.15]{Duf(97)} that $(P_n)$ 
converges a.s. to a finite random variable.
All the conditions of the second part of Theorem 1 and Corollaries 1 and 2 in \cite{Hey(77)} are satisfied, which leads to
\begin{equation}
\label{FLUC8}
\frac{M_n-M_{\alpha,\theta}}{\sqrt{\Lambda_n}} 
\underset{n\rightarrow+\infty}{\overset{\cL}{\longrightarrow}}
\mathcal{N}(0, 1)
\end{equation}
as well as to  
\begin{equation}
\label{FLUC9}
\sqrt{n^{\alpha}}(M_n-M_{\alpha,\theta}) 
\underset{n\rightarrow+\infty}{\overset{\cL}{\longrightarrow}} 
\frac{\Gamma(\alpha+\theta+1)}{\Gamma(\theta+1)}\sqrt{S_{\alpha,\theta}^\prime}
\mathcal{N}(0, 1),
\end{equation}
where $M_{\alpha,\theta}$ is the almost-sure limit of the martingale $(M_n)$ and
$S_{\alpha,\theta}^\prime$ is independent of the Gaussian $\mathcal{N}(0, 1)$ random variable
and $S_{\alpha,\theta}^\prime$ shares the same distribution as $S_{\alpha,\theta}$. 
However, the random variables $M_{\alpha,\theta}$ and $S_{\alpha,\theta}$ are tightly related through
$$M_{\alpha,\theta}=\Big( \frac{\Gamma(\alpha+\theta+1)}{\Gamma(\theta+1)}\Big)S_{\alpha,\theta}.$$
In addition, we obtain from the asymptotic behavior of the ratio of two Gamma functions that
\begin{equation}
\label{SHARPBN}
b_n= \frac{\Gamma(\alpha+\theta+1)}{\Gamma(\theta+1) n^\alpha}
\left(1-\frac{\alpha(2\theta+\alpha-1)}{2n}+O\Big(\frac{1}{n^2}\Big) \right).
\end{equation}
Finally, the Gaussian fluctuations \eqref{ANBLOCK1} and \eqref{ANBLOCK2}
follow from \eqref{FLUC8} and \eqref{FLUC9} together with the almost-sure convergences
\eqref{ASCVG}, \eqref{FLUC4}, \eqref{SHARPBN} and Slutsky's lemma, which achieves the proof of Theorem \ref{T-FLUCBLOCK}. 
\end{proof}

\subsection{The law of the iterated logarithm}

We conclude our asymptotic analysis of $K_{n}$ by establishing the law of iterated logarithm.

\begin{thm}
\label{T-LILBLOCK}
Let $K_{n}$ be the number of partition sets in a random partition of $[n]$ distributed according to the EP model with $\alpha\in(0,1)$ and $\theta>-\alpha$. Then,
\begin{equation}
\label{LILBLOCK1}
\limsup_{n \rightarrow \infty} \left(\frac{K_n-n^{\alpha}S_{\alpha,\theta}}{\sqrt{2K_n \log \log n}}\right)
= -\liminf_{n \rightarrow \infty} \left(\frac{K_n-n^{\alpha}S_{\alpha,\theta}}{\sqrt{2K_n \log \log n}}
\right)
= 1\qquad\text{a.s.}
\end{equation}
Moreover, we also have
\begin{equation*}
\label{LILBLOCK2}
\limsup_{n \rightarrow \infty} \left(\frac{K_n-n^{\alpha}S_{\alpha,\theta}}{\sqrt{2n^\alpha \log \log n}} \right)
= -\liminf_{n \rightarrow \infty} \left(\frac{K_n-n^{\alpha}S_{\alpha,\theta}}{\sqrt{2 n^\alpha \log \log n}} \right)
= \sqrt{S_{\alpha,\theta}}\qquad\text{a.s.}
\end{equation*}
In particular,
\begin{equation*}
\limsup_{n \rightarrow \infty} \frac{\big(K_n-n^{\alpha}S_{\alpha,\theta}\big)^2}{2n^\alpha \log \log n} 
=  S_{\alpha,\theta}\qquad\text{a.s.}
\end{equation*}
\end{thm}

\begin{proof}
The proof is a direct application of \cite{Hey(77)}. Using the same notation as in the proof of Theorem \ref{T-FLUCBLOCK}, we have for any $\eta >0$,
\begin{equation}
\label{LIL1}
\sum_{n=1}^{\infty}n^{\alpha/2} \dE\big[|\Delta M_{n}| \rI_{\{|\Delta M_{n}|>\eta \sqrt{n^{-\alpha}} \}}\big]\leq \frac{1}{\eta^3}\sum_{n=1}^{\infty} n^{2\alpha}\dE\big[\Delta M_{n}^4\big].
\end{equation}
Hence, we deduce from \eqref{FLUC7} and \eqref{LIL1} that it exists
some constant $C>0$ such that
\begin{equation*}
\sum_{n=1}^{\infty}n^{\alpha/2} 
\dE\big[|\Delta M_{n}| \rI_{\{|\Delta M_{n}|>\eta \sqrt{n^{-\alpha}} \}}\big]\leq C\sum_{n=1}^{\infty} \frac{1}{n^{1+\alpha}}< \infty.
\end{equation*}
Moreover, we also have for any $\delta >0$,
\begin{equation*}
\sum_{n=1}^{\infty}n^{2\alpha} 
\dE\big[\Delta M_{n}^4 \rI_{\{|\Delta M_{n}|\leq \delta \sqrt{n^{-\alpha}} \}}\big]\leq \sum_{n=1}^{\infty} n^{2\alpha} \dE\big[\Delta M_{n}^4\big]< \infty.
\end{equation*}
Furthermore, we already saw that the martingale $(P_n)$, defined by \eqref{DEFMARTPN}, 
converges a.s. to a finite random variable. Consequently, all the conditions of the second part of Theorem 1 and Corollary 2 in \cite{Hey(77)} are satisfied, which ensures that
\begin{equation}
\label{LIL2}
\limsup_{n \rightarrow \infty} \left(\frac{M_n-M_{\alpha,\theta}}{\sqrt{2\Lambda_n \log \log n}} \right)
= -\liminf_{n \rightarrow \infty} \left(\frac{M_n-M_{\alpha,\theta}}{\sqrt{2\Lambda_n \log \log n}} \right)
= 1\qquad\text{a.s.}
\end{equation}
where $M_{\alpha,\theta}$ is the almost-sure limit of the martingale $(M_n)$.
Finally, the law of iterated logarithm \eqref{LILBLOCK1}
follows from \eqref{LIL2} together with the almost-sure convergences \eqref{ASCVG} and \eqref{FLUC4}, which completes the proof of Theorem \ref{T-LILBLOCK}. 
\end{proof}


\section{Asymptotic results for the partition subsets of size $r$}\label{sec2}
We start by introducing the keystone martingale construction for $K_{r,n}$.  In particular, our approach is different from that developed in Section \ref{sec1}, since we are going to build a  martingale that will not converge almost surely to a finite random variable.  First, we consider the case $r=1$. From the sequential construction of the  EP model \citep[Proposition 9]{Pit(95)}, we have for all $n \geq 1$,
\begin{equation}\label{DEFM1}
K_{1,n+1}=K_{1,n}+\xi_{1,n+1}
\end{equation}
where the conditional distribution of the random variable $\xi_{1,n+1}$, given the $\sigma$-algebra $\mathcal{F}_{n}=\sigma(K_{1},\ldots,K_{n},K_{1,1},\ldots,K_{1,n})$, is such that
\begin{equation}
\label{DEFXIM1}
   \dP(\xi_{1,n+1}=k\,|\,\mathcal{F}_{n})=\left \{ \begin{array}{ccc}
    {p_{1,n}}  & \text{ if } & k=1, \vspace{1ex}\\[0.4cm]
    {q_{1,n}}  & \text{ if } & k=-1, \vspace{1ex}\\[0.4cm]
     {1-p_{1,n}-q_{1,n}} & \text{ if }  & k=0,
   \end{array}  \right.
\end{equation}
where
\begin{equation}\label{DEFIP1N}
p_{1,n}=\frac{\alpha K_{n}+\theta}{n+\theta}\qquad \text{and}\qquad q_{1,n}=\frac{(1-\alpha)K_{1,n}}{n+\theta}.
\end{equation}
We refer the reader to Appendix A for more details on \eqref{DEFXIM1}. From definition of the sequence $(\xi_{1,n})$, we clearly have $\E[\xi_{1,n+1}\,|\,\mathcal{F}_{n}]=p_{1,n}-q_{1,n}$. Consequently, we obtain from \eqref{DEFM1} that
\begin{equation*}
\E[K_{1,n+1}\,|\,\mathcal{F}_{n}]=\E[K_{1,n}+\xi_{1,n+1}\,|\,\mathcal{F}_{n}]=K_{1,n}+p_{1,n}-q_{1,n}\qquad\text{a.s},
\end{equation*}
which leads to
\begin{equation}
\label{CONDM1M1}
\E[K_{1,n+1}\,|\,\mathcal{F}_{n}]=\beta_{1,n}K_{1,n}+p_{1,n}\qquad \text{a.s.},
\end{equation}
where
\begin{equation}\label{DEFALFAN1}
\beta_{1,n}=1-\frac{1-\alpha}{n+\theta}=\frac{n-1+\theta+\alpha}{n+\theta}.
\end{equation}
Let $(b_{1,n})$ be the sequence defined by $b_{1,1}=1$ and for all $n\geq2$,
\begin{equation}
\label{DEFBN1}
b_{1,n}=\prod_{k=1}^{n-1}\beta_{1,k}^{-1}=\prod_{k=1}^{n-1}\left(\frac{k+\theta}{k-1+\theta+\alpha}\right)=\frac{(\theta+1)^{(n-1)}}{(\alpha+\theta)^{(n-1)}}.
\end{equation}
Now, we introduce the sequence of random variables $(M_{1,n})$ that is defined, for all $n\geq1$, by
\begin{equation}\label{DEFMN_M}
M_{1,n}=b_{1,n}K_{1,n}-A_{1,n},
\end{equation}
where
\begin{equation}\label{DEFMN_A}
A_{1,n}=\sum_{k=1}^{n-1}b_{1,k+1}\left(\frac{\alpha K_{k}+\theta}{k+\theta}\right).
\end{equation}
One can observe that $A_{1,n+1}=A_{1,n}+p_{1,n}b_{1,n+1}$ and $b_{1,n}=\beta_{1,n}b_{1,n+1}$.
Hence, we have from \eqref{CONDM1M1} and \eqref{DEFALFAN1} that for all $n\geq1$,
\begin{align*}
\E[M_{1,n+1}\,|\,\mathcal{F}_{n}]&=b_{1,n+1}\E[K_{1,n+1}\,|\,\mathcal{F}_{n}]-A_{1,n+1}\\
&=b_{1,n+1}(\beta_{1,n}K_{1,n}+p_{1,n})-A_{1,n}-p_{1,n}b_{1,n+1}\\
&=\beta_{1,n}b_{1,n+1}K_{1,n}-A_{1,n}\\
&=b_{1,n}K_{1,n}-A_{1,n}\\
&=M_{1,n}
\end{align*}
almost surely. Moreover, $(M_{1,n})$ is square integrable as $K_{1,n}\leq K_{n}\leq n$. Consequently, $(M_{1,n})$ is a locally square integrable martingale. Because of \eqref{DEFXIM1}, it is not hard to compute all the higher order conditional moments of $(M_{1,n})$. For example, as
\begin{displaymath}
\Delta M_{1,n+1}=M_{1,n+1}-M_{1,n}=b_{1,n+1}(K_{1,n+1}-\E[K_{1,n+1}\,|\,\mathcal{F}_{n}])
\end{displaymath}
we clearly have from \eqref{DEFM1} and \eqref{CONDM1M1} that
\begin{displaymath}
\E[(\Delta M_{1,n+1})^{2}\,|\,\mathcal{F}_{n}]=b_{1,n+1}^{2}(\E[K_{1,n+1}^{2}\,|\,\mathcal{F}_{n}]-(\E[K_{1,n+1}\,|\,\mathcal{F}_{n}])^{2})
\end{displaymath}
with 
\begin{equation*}
\E[K_{1,n+1}^{2}\,|\,\mathcal{F}_{n}]=K_{1,n}^{2}+2K_{1,n}(p_{1,n}-q_{1,n})+p_{1,n}+q_{1,n}
\end{equation*}
and
\begin{equation*}
(\E[K_{1,n+1}\,|\,\mathcal{F}_{n}])^{2}=(K_{1,n}+p_{1,n}-q_{1,n})^{2}
\end{equation*}
almost surely, which reduces to
\begin{equation}
\label{COND2MN_M}
\E[(\Delta M_{1,n+1})^{2}\,|\,\mathcal{F}_{n}]
=b_{1,n+1}^{2}\big(p_{1,n}+q_{1,n} -(p_{1,n}-q_{1,n})^{2}\big)\hspace{1cm} \text{a.s.}
\end{equation}
The above calculations extend easily to the case $r\geq2$. From the sequential construction of the EP model \citep[Proposition 9]{Pit(95)}, we have for all $n \geq 1$,
\begin{equation}\label{DEFMr}
K_{r,n+1}=K_{r,n}+\xi_{r,n+1}
\end{equation}
where the conditional distribution of the random variable $\xi_{r,n+1}$, given the $\sigma$-algebra $\mathcal{F}_{n}=\sigma(K_{r-1,1},\ldots,K_{r-1,n},K_{r,1},\ldots,K_{r,n})$, is such that
\begin{equation}\label{DEFXIMr}
   \dP(\xi_{r,n+1}=k\,|\,\mathcal{F}_{n})=\left \{ \begin{array}{ccc}
    {p_{r,n}}  & \text{ if } & k=1, \vspace{1ex}\\[0.4cm]
    {q_{r,n}}  & \text{ if } & k=-1, \vspace{1ex}\\[0.4cm]
     {1-p_{r,n}-q_{r,n}} & \text{ if }  & k=0,
   \end{array}  \right.
\end{equation}
where
\begin{equation}\label{DEFIPrN}
p_{r,n}=\frac{(r-1-\alpha)K_{r-1,n}}{n+\theta}\qquad\text{and}\qquad q_{r,n}=\frac{(r-\alpha)K_{r,n}}{n+\theta}.
\end{equation}
We refer the reader to Appendix A for more details on \eqref{DEFXIMr}. As before, $\E[\xi_{r,n+1}\,|\,\mathcal{F}_{n}]=p_{r,n}-q_{r,n}$, which implies that
\begin{equation}
\label{CONDM1Mr}
\E[K_{r,n+1}\,|\,\mathcal{F}_{n}]=\beta_{r,n}K_{r,n}+p_{r,n}\qquad\text{a.s.},
\end{equation}
where
\begin{equation}\label{DEFALFANr}
\beta_{r,n}=1-\frac{r-\alpha}{n+\theta}=\frac{n-r+\theta+\alpha}{n+\theta}.
\end{equation}
Let $(b_{r,n})$ be the sequence defined by $b_{r,1}=1$ and for all $n\geq2$,
\begin{equation}
\label{DEFBNr}
b_{r,n}=\prod_{k=r}^{n-1}\beta_{r,k}^{-1}=\prod_{k=r}^{n-1}\left(\frac{k+\theta}{k-r+\theta+\alpha}\right)=\frac{(\theta+1)^{(n-1)}}{(\alpha+\theta)^{(n-r)}}.
\end{equation}
Hereafter, we introduce the sequence of random variables $(M_{r,n})$ that is defined, for all $n\geq1$, by
\begin{equation}\label{DEFMN_Mr}
M_{r,n}=b_{r,n}K_{r,n}-A_{r,n},
\end{equation}
where
\begin{equation}\label{DEFMN_Ar}
A_{r,n}=\sum_{k=1}^{n-1}b_{r,k+1}\left(\frac{r-1-\alpha}{k+\theta}\right)K_{r-1,k}.
\end{equation}
Since 
$A_{r,n+1}=A_{r,n}+p_{r,n}b_{r,n+1}$
and 
$b_{r,n}=\beta_{r,n}b_{r,n+1}$, we have from \eqref{CONDM1Mr} and \eqref{DEFALFANr} that for all $n\geq1$,
\begin{align*}
\E[M_{r,n+1}\,|\,\mathcal{F}_{n}]&=b_{r,n+1}\E[K_{r,n+1}\,|\,\mathcal{F}_{n}]-A_{r,n+1}\\
&=b_{r,n+1}(\beta_{r,n}K_{r,n}+p_{r,n})-A_{r,n}-p_{r,n}b_{r,n+1}\\
&=\beta_{r,n}b_{r,n+1}K_{r,n}-A_{r,n}\\
&=b_{r,n}K_{r,n}-A_{r,n}\\
&=M_{r,n}
\end{align*}
almost surely. In addition, $(M_{r,n})$ is square integrable as $K_{r,n}\leq K_{n}\leq n$. Consequently, $(M_{r,n})$ is a locally square integrable martingale. Because of \eqref{DEFXIMr}, it is not hard to compute all the higher order conditional moments of $(M_{r,n})$, as it was previously done for $r=1$.

\subsection{An alternative approach for the almost-sure convergence}

Based on the above martingale construction, we present an alternative proof of \eqref{as_ep_2}, relying on the strong law of large numbers for martingale. Our proof is more rigorous than \citet[Lemma 3.11]{Pit(06)}, which contains only a sketch for the proof of \eqref{as_ep_2}. 

\begin{thm}
\label{T-ASCVG_M}
Let $K_{r,n}$ be the number of partition subsets of size $r$ in a random partition of $[n]$ distributed according to the EP model with $\alpha\in(0,1)$ and $\theta>-\alpha$. Then, for all $r\geq 1$
\begin{equation}
\label{ASCVG_M}
\lim_{n \rightarrow+\infty}
\frac{K_{r,n}}{n^{\alpha}}=p_{\alpha}(r)S_{\alpha,\theta}\qquad\text{a.s.}
\end{equation}
where 
\begin{displaymath}
p_{\alpha}(r)=\frac{\alpha(1-\alpha)^{(r-1)}}{r!}
\end{displaymath}
and $S_{\alpha,\theta}$ is a positive and almost surely finite random variable whose 
distribution has probability density function \eqref{pitman_div}.
This convergence holds in $\mathbb{L}^p$ for any 
integer $p\geq 1$, 
\begin{equation}
\label{ASCVGLP_M}
 \lim_{n \rightarrow+\infty} \dE\left[ \left| \frac{K_{r,n}}{n^{\alpha}} -p_{\alpha}(r)S_{\alpha,\theta} \right|^p \right]=0.
\end{equation}
\end{thm}

\begin{proof}
The proof is by induction on $r\geq1$. In particular, we start by considering the case $r=1$. We already saw that $(M_{1,n})$ is a locally square integrable martingale. We deduce from \eqref{COND2MN_M} that its predictable quadratic variation is given by
\begin{equation*}
\langle M_{1} \rangle_n  = \sum_{k=1}^{n-1} b_{1,k+1}^{2} \big(p_{1,k}+q_{1,k} -(p_{1,k}-q_{1,k})^{2}\big).
\end{equation*}
Consequently,
\begin{displaymath}
\langle M_{1} \rangle_n\leq \sum_{k=1}^{n-1}b_{1,k+1}^{2}\big(p_{1,k}+q_{1,k}\big).
\end{displaymath}
On the one hand, it follows directly from the definition of $p_{1,n}$ in \eqref{DEFIP1N} together with the almost-sure convergence \eqref{ASCVG} that
\begin{equation}\label{LIM1_P}
\lim_{n\rightarrow+\infty}n^{1-\alpha}p_{1,n}=\alpha S_{\alpha,\theta}\qquad\text{a.s.}
\end{equation}
On the other hand, we have from \eqref{DEFBN1} that 
\begin{equation}\label{LIM1_A}
\lim_{n\rightarrow+\infty}\frac{b_{1,n}}{n^{1-\alpha}}=\frac{\Gamma(\alpha+\theta)}{\Gamma(\theta+1)}.
\end{equation}
Accordingly, as $(n+\theta)q_{1,n}\leq K_n$, we deduce from \eqref{LIM1_P} and
\eqref{LIM1_A}
 through a direct application of Toeplitz lemma that
\begin{displaymath}
\langle M_{1} \rangle_n=O\left(\sum_{k=1}^{n}k^{1-\alpha}\right)=O(n^{2-\alpha})\qquad \text{a.s.}
\end{displaymath}
Therefore, it follows from the strong law of large numbers for martingales given by 
\citep[Theorem 1.3.24]{Duf(97)} that
\begin{displaymath}
(M_{1,n})^{2}=O(n^{2-\alpha}\log n)\qquad \text{a.s.},
\end{displaymath}
which leads to
\begin{displaymath}
(b_{1,n}K_{1,n}-A_{1,n})^{2}=O(n^{2-\alpha}\log n)\qquad \text{a.s.}
\end{displaymath}
It clearly implies that
\begin{equation}
\label{LIM1KA}
\lim_{n\rightarrow+\infty}\frac{b_{1,n}K_{1,n}-A_{1,n}}{n}=0\qquad\text{a.s.}
\end{equation}
Furthermore, by combining \eqref{LIM1_P} and \eqref{LIM1_A} together with \eqref{DEFMN_A}, we find that
\begin{equation}
\label{LIM1AFIN}
\lim_{n\rightarrow+\infty}\frac{A_{1,n}}{n}=\frac{\alpha\Gamma(\alpha+\theta)}{\Gamma(\theta+1)}S_{\alpha,\theta}\qquad \text{a.s.}
\end{equation}
Consequently, we obtain from \eqref{LIM1KA} and \eqref{LIM1AFIN} that
\begin{displaymath}
\lim_{n\rightarrow+\infty}\frac{b_{1,n}K_{1,n}}{n}=\lim_{n\rightarrow+\infty}\frac{A_{1,n}}{n}=\frac{\alpha\Gamma(\alpha+\theta)}{\Gamma(\theta+1)}S_{\alpha,\theta}\qquad \text{a.s.}
\end{displaymath}
Then, we deduce from \eqref{LIM1_A} that
\begin{displaymath}
\lim_{n\rightarrow+\infty}\frac{K_{1,n}}{n^{\alpha}}=\alpha S_{\alpha,\theta}\qquad \text{a.s.}
\end{displaymath}
This completes the proof for $r=1$. For $r\geq2$, we shall proceed by induction, assuming that 
\begin{equation}\label{induct_h}
\lim_{n\rightarrow+\infty}\frac{K_{r-1,n}}{n^{\alpha}}=\frac{\alpha(1-\alpha)^{(r-2)}}{(r-1)!}S_{\alpha,\theta}\qquad \text{a.s.}
\end{equation}
We already saw that $(M_{r,n})$ is a locally square integrable martingale. As in the case
$r=1$, one can easily see that its predictable quadratic variation is given by
\begin{equation}
\label{COND2MN_Mr}
\langle M_{r} \rangle_n=\sum_{k=1}^{n-1}b_{r,k+1}^{2}
\big(p_{r,k}+q_{r,k} -(p_{r,k}-q_{r,k})^{2}\big).
\end{equation}
On the one hand, it follows from the definition of $p_{r,n}$ in \eqref{DEFIPrN} together with the induction hypothesis \eqref{induct_h} that
\begin{eqnarray}
\lim_{n\rightarrow+\infty}n^{1-\alpha}p_{r,n}&=&\frac{(r-1-\alpha)\alpha(1-\alpha)^{(r-2)}}{(r-1)!}
S_{\alpha,\theta}\qquad \text{a.s.}
\nonumber \\
&=&\frac{\alpha(1-\alpha)^{(r-1)}}{(r-1)!}S_{\alpha,\theta}\qquad \text{a.s.}
\label{LIMr_P}
\end{eqnarray}
On the other hand, we have from \eqref{DEFBNr} that 
\begin{equation}\label{LIMr_A}
\lim_{n\rightarrow+\infty}\frac{b_{r,n}}{n^{r-\alpha}}=\frac{\Gamma(\alpha+\theta)}{\Gamma(\theta+1)}.
\end{equation}
Accordingly, as $(n+\theta)q_{r,n}\leq r K_{n}$, we deduce 
from \eqref{COND2MN_Mr}, \eqref{LIMr_P} and \eqref{LIMr_A} together with Toeplitz lemma that
\begin{displaymath}
\langle M_{r} \rangle_n=O\left(\sum_{k=1}^{n}k^{2r-1-\alpha}\right)=O(n^{2r-\alpha})\qquad \text{a.s.}
\end{displaymath}
Therefore, it follows from the strong law of large numbers for martingales that
\begin{displaymath}
(M_{r,n})^{2}=O(n^{2r-\alpha}\log n)\qquad\text{a.s.},
\end{displaymath}
which leads to
\begin{displaymath}
(b_{r,n}K_{r,n}-A_{r,n})^{2}=O(n^{2r-\alpha}\log n)\qquad \text{a.s.}
\end{displaymath}
As before, it implies that
\begin{equation}
\label{LIMKAr}
\lim_{n\rightarrow+\infty}\frac{b_{r,n}K_{r,n}-A_{r,n}}{n^{r}}=0\qquad \text{a.s.}
\end{equation}
Hereafter, by combining \eqref{LIMr_P} and \eqref{LIMr_A} together with \eqref{DEFMN_Ar}, we obtain that
\begin{equation}
\label{LIMAFINr}
\lim_{n\rightarrow+\infty}\frac{A_{r,n}}{n^{r}}=\frac{\Gamma(\alpha+\theta)}{\Gamma(\theta+1)}\frac{\alpha(1-\alpha)^{(r-1)}}{r!}S_{\alpha,\theta}\qquad\text{a.s.}
\end{equation}
Finally, we find from \eqref{LIMKAr} and \eqref{LIMAFINr} that
\begin{displaymath}
\lim_{n\rightarrow+\infty}\frac{b_{r,n}K_{r,n}}{n^{r}}=\lim_{n\rightarrow+\infty}\frac{A_{r,n}}{n^{r}}=\frac{\Gamma(\alpha+\theta)}{\Gamma(\theta+1)}\frac{\alpha(1-\alpha)^{(r-1)}}{r!}S_{\alpha,\theta}\qquad\text{a.s.}
\end{displaymath}
which ensures via \eqref{LIM1_A} that
\begin{displaymath}
\lim_{n\rightarrow+\infty}\frac{K_{r,n}}{n^{\alpha}}=\frac{\alpha(1-\alpha)^{(r-1)}}{r!} S_{\alpha,\theta}\qquad \text{a.s.}
\end{displaymath}
It only remains to prove the convergence in $\mathbb{L}^p$ given by \eqref{ASCVGLP_M}.
We have from \citet[Proposition 1]{Fav(13)} that for any integer $p\geq 1$,
\begin{align}\label{equazione_favaro}
\E[(K_{r,n})_{(p)}]
&=(p_r(\alpha))^{p}\frac{n!}{(n-rp)!}\sum_{k=0}^{n-pr}\frac{\left(\frac{\theta}{\alpha}\right)^{(k+p)}}{(\theta)^{(n)}}\mathscr{C}(n-pr,k;\alpha)\\
&=(p_r(\alpha))^{p}\frac{n!}{(n-rp)!}\left(\frac{\theta}{\alpha}\right)^{(p)}\sum_{k=0}^{n-pr}\frac{\left(\frac{\theta}{\alpha}+p\right)^{(k)}}{(\theta)^{(n)}}\mathscr{C}(n-pr,k;\alpha) \notag
\end{align}
where $\mathscr{C}(n,k;\alpha)$ is the generalized factorial coefficient \citep[Chapter 2]{Cha(07)},
\begin{displaymath}
\mathscr{C}(n,k;\alpha)=\frac{1}{k!}\sum_{j=1}^{k}(-1)^{j}{k\choose j}(-j\alpha)^{(n)}.
\end{displaymath}
An application of \citep[Equation 2.46]{Cha(07)} allows to solve the summation over $k$ in \eqref{equazione_favaro}. More precisely,
\begin{align*}
\E[(K_{r,n})_{(p)}]=(p_r(\alpha))^{p}\frac{n!}{(n-rp)!}\left(\frac{\theta}{\alpha}\right)^{(p)}\frac{(\theta+\alpha p)^{(n-rp)}}{(\theta)^{(n)}}.
\end{align*}
Consequently,
\begin{equation}
\label{LP_M1}
\lim_{n \rightarrow+\infty} \frac{1}{n^{\alpha p}}\E[(K_{r,n})_{(p)}]=
\big(p_r(\alpha)\big)^{p} \Big(\frac{\theta}{\alpha}\Big)^{\!(p)\!}
\frac{\Gamma(\theta +1)}{\theta \Gamma(\alpha p +\theta)} \Big(\frac{\theta}{\alpha}\Big)^{\!(p)\!}
\end{equation}
However, as in identity \eqref{LIMMKN2}, we have
\begin{equation}
\label{LP_M2}
\dE[K_{r,n}^p]=\sum_{k=0}^p \left\{ 
\begin{matrix}
p \\k 
\end{matrix}
\right\}
\dE[(K_{r,n})_{(k)}].
\end{equation}
Then, we deduce from \eqref{LP_M1} and \eqref{LP_M2} that
\begin{equation}
\label{LP_M3}
\lim_{n \rightarrow+\infty} \frac{1}{n^{\alpha p}}\dE[K_{r,n}^p]=
\big(p_r(\alpha)\big)^{p} 
\frac{\Gamma(\theta +1)}{\theta \Gamma(\alpha p +\theta)}\Big(\frac{\theta}{\alpha}\Big)^{\!(p)\!}= \E\big[\big(p_r(\alpha) S_{\alpha,\theta}\big)^{p}\big].
\end{equation}
Finally, \eqref{ASCVGLP_M}
follows from \eqref{ASCVG_M} and \eqref{LP_M3} together with Riesz-Scheff\'e theorem,
which completes the proof of Theorem \ref{T-ASCVG_M}.

\end{proof}

\begin{remark}
An alternative proof of the convergence in $\mathbb{L}^{2}$ given by \eqref{T-ASCVG_M}, which follows directly from the Poisson process construction of the EP model \citep[Proposition 9]{Pit(03)} can be found in Appendix B. See also \citet[Chapter 4]{Pit(06)}. 
\end{remark}

\subsection{The Gaussian fluctuation}

We shall now focus on a Gaussian fluctuation of $K_{r,n}$, properly normalized,  around an estimator of its almost-sure limit  $p_{\alpha}(r)S_{\alpha,\theta}$.

\begin{thm}
\label{T-FLUCBLOCK_M}
Let $K_{r,n}$ be the number of partition subsets of size $r$ in a random partition of $[n]$ distributed according to the EP model with $\alpha\in(0,1)$ and $\theta>-\alpha$. Then, for all $r\geq 1$
\begin{equation}
\label{ANBLOCK_M1}
\frac{1}{\sqrt{K_{r,n}}} \left(K_{r,n}- \frac{A_{r,n}}{b_{r,n}} \right)
\underset{n\rightarrow+\infty}{\overset{\cL}{\longrightarrow}}
\mathcal{N}(0,1),
\end{equation}
where $\mathcal{N}(0,1)$ denotes a standard Gaussian random variable. Moreover, we also
have for all $r\geq 1$ 
\begin{equation}
\label{ANBLOCK_M2}
\sqrt{n^\alpha} \left(\frac{K_{r,n}}{n^{\alpha}}- \frac{A_{r,n}}{b_{r,n}n^\alpha} \right)
\underset{n\rightarrow+\infty}{\overset{\cL}{\longrightarrow}}
\sqrt{p_\alpha(r) S_{\alpha,\theta}^\prime}\, \mathcal{N}(0,1),
\end{equation}
where 
$S_{\alpha,\theta}^\prime$ is a random variable independent of $\mathcal{N}(0,1)$ and sharing the same distribution as $S_{\alpha,\theta}$.
\end{thm}

\begin{proof}
We are going to carry out the proof for any $r\geq1$. It follows from \eqref{ASCVG_M},
\eqref{COND2MN_Mr},  \eqref{LIMr_P} and \eqref{LIMr_A} together with Toeplitz lemma that the 
predictable quadratic variation $\langle M_{r} \rangle_n$ satisfies
\begin{equation}
\label{ASCVGPQVr1}
\lim_{n\rightarrow+\infty}\frac{\langle M_{r} \rangle_n}{n^{2r-\alpha}}
=p_\alpha(r) 
  \left( \frac{\Gamma(\alpha+\theta)}{\Gamma(\theta+1)}\right)^2S_{\alpha,\theta}\hspace{1cm} \text{a.s.}
\end{equation}
Moreover, we have from \eqref{ASCVGLP_M} with $p=1$ and $p=2$ that
\begin{equation}
\label{LIMPQVPr1}
\lim_{n\rightarrow+\infty}n^{1-\alpha} \dE[p_{r,n}]=(r-1-\alpha)p_\alpha(r-1) \dE[S_{\alpha, \theta}]
\end{equation}
and
\begin{equation}
\label{LIMPQVPr2}
\lim_{n\rightarrow+\infty}n^{2-2\alpha} \dE[p_{r,n}^2]=(r-1-\alpha)^2p^2_\alpha(r-1) \dE[S_{\alpha, \theta}^2].
\end{equation}
In addition, we also have from \eqref{ASCVGLP_M} with $p=1$ and $p=2$ that
\begin{equation}
\label{LIMPQVQr1}
\lim_{n\rightarrow+\infty}n^{1-\alpha} \dE[q_{r,n}]=(r-\alpha) p_\alpha(r) \dE[S_{\alpha, \theta}]
\end{equation}
and
\begin{equation}
\label{LIMPQVQr2}
\lim_{n\rightarrow+\infty}n^{2-2\alpha} \dE[q_{r,n}^2]=(r-\alpha)^2 p^2_\alpha(r) \dE[S_{\alpha, \theta}^2].
\end{equation}
Therefore, denote $s_{r,n}^2=\dE[\langle M_{r} \rangle_n]$. We deduce from \eqref{COND2MN_Mr} and \eqref{LIMr_A} together with \eqref{LIMPQVPr1}, \eqref{LIMPQVPr2}, \eqref{LIMPQVQr1} and \eqref{LIMPQVQr2} that
\begin{eqnarray}
\lim_{n\rightarrow+\infty}\frac{s_{r,n}^2}{n^{2r-\alpha}}
&=&\lim_{n\rightarrow+\infty}\frac{1}{n^{2r-\alpha}}
\sum_{k=1}^{n-1} b_{r,k+1}^2 \big( \dE[p_{r,k}]+\dE[q_{r,k}]\big), \notag \\
&=& 
\frac{p_\alpha(r)}{\alpha}
\left(\frac{\Gamma(\alpha+\theta)}{\Gamma(\theta+1)}\right).
\label{ASCVGPQVr2}
\end{eqnarray}
Hence, we obtain from \eqref{ASCVGPQVr1} and \eqref{ASCVGPQVr2} that
\begin{equation}
\label{ASCVGPQVr3}
\lim_{n\rightarrow+\infty}\frac{\langle M_{r} \rangle_n}{s_{r,n}^2}
= \alpha \left( \frac{\Gamma(\alpha+\theta)}{\Gamma(\theta+1)}\right)S_{\alpha,\theta}\hspace{1cm} \text{a.s.}
\end{equation} 
We are now going to show that for any $\eta >0$,
\begin{equation}\label{ASCVGPQVr4}
\lim_{n \rightarrow+\infty} \frac{1}{n^{2r-\alpha}}\sum_{k=1}^{n}\dE\big[\Delta M_{r,k}^2 \rI_{\{|\Delta M_{r, k}|>\eta s_{r,n} \}}\big]=0.
\end{equation}
We clearly have  for any $\eta >0$,
$$
\sum_{k=1}^{n}\dE\big[\Delta M_{r,k}^2 \rI_{\{|\Delta M_{r, k}|>\eta s_{r,n} \}}\big]
\leq 
\frac{1}{\eta^2 s_{r,n}^2}\sum_{k=1}^{n}\dE\big[\Delta M_{r, k}^4\big].
$$
According to \eqref{ASCVGPQVr2}, in order to prove \eqref{ASCVGPQVr4}, 
it is only necessary to show that
\begin{equation}\label{ASCVGPQVr5}
\lim_{n \rightarrow+\infty} \frac{1}{n^{4r-2\alpha}}\sum_{k=1}^{n}
\dE\big[\Delta M_{r, k}^4\big]=0.
\end{equation}
As in the proof of \eqref{DECMN4COND}, it follows from tedious but straightforward calculations that for all $n \geq 1$,
\begin{equation}
\label{ASCVGPQVr6}
\dE\big[\Delta M_{r,n+1}^4 | \cF_n\big]=b_{r,n+1}^4 \sum_{\ell=1}^4 e_{r,n}(\ell)
\end{equation}
where
\begin{align*}
e_{r,n}(1)&=(p_{r,n}+q_{r,n}), \\
e_{r,n}(2)&=-4(p_{r,n}-q_{r,n})^2, \\
e_{r,n}(3)&=6(p_{r,n}+q_{r,n})(p_{r,n}-q_{r,n})^2, \\
e_{r,n}(4)&=-3(p_{r,n}-q_{r,n})^4.
\end{align*}
Then, it follows from \eqref{ASCVGPQVr6} that it exists some constant $C>0$ such that for $n$ large enough,
\begin{equation}
\label{ASCVGPQVr7}
\dE\big[\Delta M_{r,n+1}^4 | \cF_n\big] \leq  C b_{r,n+1}^4 \big(p_{r,n}+q_{r,n}\big).
\end{equation}
Taking the expectation on both sides of \eqref{ASCVGPQVr7}, we obtain that for $n$ large enough,
\begin{equation}
\label{ASCVGPQVr8}
\dE\big[\Delta M_{r,n+1}^4 \big]\leq C b_{r,n+1}^4 \big(\dE[p_{r,n}]+\dE[q_{r,n}]\big).
\end{equation}
Consequently, we deduce from \eqref{LIMr_A}, \eqref{LIMPQVPr1}, \eqref{LIMPQVQr1} and \eqref{ASCVGPQVr8} that
\begin{equation*}
\sum_{k=1}^{n}\dE\big[\Delta M_{r, k}^4\big]=O\left( \sum_{k=1}^n k^{4r-1-3\alpha} \right)=O(n^{4r-3\alpha})
\end{equation*}
which leads to \eqref{ASCVGPQVr5}. Hence, we find from Corollary 1 in \cite{Hey(77)} that
\begin{equation}
\label{ASCVGPQVr9}
\frac{M_{r,n}}{\sqrt{\langle M_{r} \rangle_n}} \stackrel{\mathcal{L}}{\longrightarrow} \mathcal{N}(0, 1),
\end{equation}
and
\begin{equation}
\label{ASCVGPQVr10}
\frac{M_{r,n}}{s_{r,n}}
\underset{n\rightarrow+\infty}{\overset{\cL}{\longrightarrow}}
\sqrt{ \frac{ \alpha \Gamma(\alpha+\theta)}{\Gamma(\theta+1)}}\,\sqrt{S_{\alpha,\theta}^\prime} \mathcal{N}(0,1)
\end{equation}
where $S_{\alpha,\theta}^\prime$ stands for a random variable independent of $\mathcal{N}(0,1)$ and sharing the same distribution as $S_{\alpha,\theta}$.
Finally, as $M_{r,n}=b_{r,n}K_{r,n} -A_{r,n}$, the Gaussian fluctuations \eqref{ANBLOCK_M1} and
\eqref{ANBLOCK_M2} follow from \eqref{LIMr_A}, \eqref{ASCVGPQVr9} and \eqref{ASCVGPQVr10}
together with the almost-sure convergences
\eqref{ASCVG_M}, \eqref{ASCVGPQVr1}, \eqref{ASCVGPQVr3} and Slutsky's lemma, 
 which achieves the proof of Theorem \ref{T-FLUCBLOCK_M}.
\end{proof}

\subsection{The law of the iterated logarithm}

We conclude our asymptotic analysis of $K_{r,n}$ by establishing the law of the iterated logarithm.

\begin{thm}
\label{T-LILBLOCK_M}
Let $K_{r,n}$ be the number of partition subsets of size $r$ in a random partition of $[n]$ distributed according to the EP model with $\alpha\in(0,1)$ and $\theta>-\alpha$. Then, for all $r\geq 1$
\begin{eqnarray}
\label{LILBLOCK1_M}
 & & \hspace{-4ex} \limsup_{n \rightarrow \infty}
\sqrt{\frac{n^\alpha}{2 \log \log n}} \left(\frac{K_{r,n}}{n^{\alpha}}- \frac{A_{r,n}}{b_{r,n}n^\alpha} \right)  \\
&=& -\liminf_{n \rightarrow \infty} \sqrt{\frac{n^\alpha}{2 \log \log n}} \left(\frac{K_{r,n}}{n^{\alpha}}- \frac{A_{r,n}}{b_{r,n}n^\alpha} \right) \nonumber \\
&=& \sqrt{p_{\alpha}(r) S_{\alpha,\theta}}\qquad\text{a.s.}
\nonumber
\end{eqnarray}
In particular,
\begin{equation*}
 \limsup_{n \rightarrow \infty}
\frac{n^\alpha}{2 \log \log n} \left(\frac{K_{r,n}}{n^{\alpha}}- \frac{A_{r,n}}{b_{r,n}n^\alpha} \right)^2 =p_{\alpha}(r) S_{\alpha,\theta}\qquad\text{a.s.}
\end{equation*}
\end{thm}

\begin{proof} As in Section \ref{sec1}, the proof is a direct application of \cite{Hey(77)}. 
We have for any $\eta >0$,
\begin{equation}
\label{LILBLOCK1PR}
\sum_{n=1}^{\infty}\frac{n^{\alpha/2}}{n^r} \dE\big[|\Delta M_{r,n}| 
\rI_{\{ \sqrt{n^{\alpha}}  |\Delta M_{r,n}|>\eta n^r \}}\big]\leq \frac{1}{\eta^3}\sum_{n=1}^{\infty} \frac{n^{2\alpha}}{n^{4r}}\dE\big[\Delta M_{r,n}^4\big].
\end{equation}
Hence, we deduce from  \eqref{ASCVGPQVr8} and \eqref{LILBLOCK1PR} together with \eqref{LIMr_A}, \eqref{LIMPQVPr1} and \eqref{LIMPQVQr1} that it exists
some constant $C>0$ such that
\begin{equation*}
\sum_{n=1}^{\infty}\frac{n^{\alpha/2}}{n^r} \dE\big[|\Delta M_{r,n}| 
\rI_{\{ \sqrt{n^{\alpha}}  |\Delta M_{r,n}|>\eta n^r \}}\big]\leq C\sum_{n=1}^{\infty} \frac{1}{n^{1+\alpha}}< \infty.
\end{equation*}
Moreover, we also have for any $\delta >0$,
\begin{equation*}
\sum_{n=1}^{\infty} \frac{n^{2\alpha}}{n^{4r}}
\dE\big[\Delta M_{r,n}^4 \rI_{\{\sqrt{n^{\alpha}}  |\Delta M_{r,n}|\leq \delta n^r \}}\big]\leq \sum_{n=1}^{\infty} \frac{n^{2\alpha}}{n^{4r}}\dE\big[\Delta M_{r,n}^4\big]< \infty.
\end{equation*}
Hereafter, let $(P_{r,n})$ be the martingale defined by
\begin{equation}
\label{DEFMARTPNBLOCK}
P_{r,n+1}=\sum_{k=1}^n \frac{k^{\alpha}}{k^{2r}} \big(\Delta M_{r,k+1}^2 - \dE[\Delta M_{r,k+1}^2\,|\, \cF_{k}]\big).
\end{equation}
Its predictable quadratic variation is given by
$$
\langle P_r \rangle_{n+1} = \sum_{k=1}^n \frac{k^{2\alpha}}{k^{4r}}
\big(\dE[\Delta M_{k+1}^4\,|\, \cF_{k}]- (\dE[\Delta M_{k+1}^2\,|\,\cF_{k}])^2\big).
$$
Hence, we obtain from \eqref{ASCVGPQVr7} that it exists some constant $C>0$ such that for $n$ large enough,
\begin{equation}
\label{CVGMARTCNBLOCK}
\langle P_r \rangle_{n} \leq C \sum_{k=1}^n \frac{k^{2\alpha}}{k^{4r}} 
b_{r,k+1}^4 \big(p_{r,k}+q_{r,k}\big).
\end{equation}
Consequently, we deduce from \eqref{COND2MN_Mr}, \eqref{LIMr_P} and \eqref{CVGMARTCNBLOCK} that $\langle P_r \rangle_n$ converges 
a.s. to a finite random variable. Then, it follows from the strong law of large numbers 
for martingales that $(P_{r,n})$ converges a.s. to a finite random variable.
Therefore, all the conditions of the second part of Theorem 1 and Corollary 2 in \cite{Hey(77)} are satisfied, which ensures that
\begin{equation}
\label{LILBLOCK2PR}
\limsup_{n \rightarrow \infty} \frac{M_{r,n}}{\sqrt{2 \langle M_r \rangle_{n} \log \log n}}
= -\liminf_{n \rightarrow \infty} \frac{M_{r,n}}{\sqrt{2 \langle M_r \rangle_{n} \log \log n}}
= 1\quad\text{a.s.}
\end{equation}
Finally, as $M_{r,n}=b_{r,n}K_{r,n} -A_{r,n}$, the law of iterated logarithm \eqref{LILBLOCK1_M}
follows from \eqref{LIMr_A} together with the almost-sure convergence \eqref{ASCVGPQVr1} and \eqref{LILBLOCK2PR}, which completes the proof of Theorem \ref{T-LILBLOCK_M}. 
\end{proof}


\section{Discussion}\label{sec3}

In this paper, we presented a unified and comprehensive treatment of the large $n$ asymptotic behaviour of $K_{n}$ and $K_{r,m}$ in the EP model with $\alpha\in(0,1)$ and $\theta>-\alpha$. By means of a novel martingale construction for $K_{n}$ and $K_{r,m}$, we obtained alternative, and rigorous, proofs of the almost-sure convergence of $K_{n}$ and $K_{r,n}$, and also covered the gap of Gaussian fluctuations. We argue that our martingale approach may be further investigated to refine Theorem \ref{T-ASCVG} and Theorem \ref{T-ASCVG_M} in terms of Berry–Esseen theorems, as well as sharp large deviations and concentration inequalities for $K_{n}$ and $K_{r,n}$. We refer to \citet{Fen(98)}, \citet{Fav(15)}, \citet{Fav(18)}, \citet{Dol(20)} and \citet{Oli(22)} for early results along these lines of research. Investigating the large $n$ asymptotic behaviour of functions that involve both $K_{n}$ and the $K_{r,n}$'s  is also a promising direction for future research. Our almost-sure limits imply
\begin{displaymath}
\lim_{n\rightarrow+\infty}\frac{K_{1,n}}{K_{n}}=\alpha\qquad\text{a.s.},
\end{displaymath}
meaning that $K_{1,n}/K_{n}$ is a consistent estimator for the parameter $\alpha$. Establishing a Gaussian fluctuation for $K_{1,n}/K_{n}$ is an interesting open problem, with potential applications in the context of Bayesian nonparametric statistics \citep{Bal(23),Fav(23)}. A further problem of interest is to establish an almost-sure limit and a Gaussian fluctuation for $(K_{1,n},K_{2,n},\ldots)$, thus providing a counterpart of \eqref{w_e_2}.


\section*{Appendix A. Sequential or generative construction of the EP model}
\renewcommand{\thesection}{\Alph{section}}
\renewcommand{\theequation}{\thesection.\arabic{equation}}
\setcounter{section}{1}
\setcounter{equation}{0}
\setcounter{thm}{0}

The purpose of this appendix is to recall the sequential construction of the EP model given in the seminal work of \citet[Proposition 9]{Pit(95)}. For any $\alpha\in[0,1)$, $\theta>-\alpha$, an exchangeable random partition of
$[n]=\{1,\ldots,n \}$ can be recursively constructed as follows: 
conditionally on the total size $K_n=k$ of the partition and on the partition subsets $\{A_1, \ldots ,A_k \}$ of corresponding sizes $(n_1, \ldots,n_k)$, the partition $[n+1]$ is an extension of $[n]$ such that the element $n+1$ is attached to subset $A_i$ for $1 \leq i \leq k$, with probability
$$
\frac{n_i- \alpha}{n+ \theta},
$$
or forms a new subset with probability
$$
\frac{\alpha k + \theta}{n+ \theta}.
$$
Since $n=n_1+ \cdots+n_k$, we clearly have
\begin{equation*}
    \frac{1}{n+\theta}\sum_{i=1}^{k} (n_i - \alpha) + \frac{ \alpha k + \theta}{n+\theta}= \frac{n-\alpha k + \alpha k + \theta}{n+ \theta}=1.
\end{equation*}
Hereafter, for $K_{n}\in\{1,\ldots,n\}$, denote by $\mathbf{N}_{n}=(N_{1,n},\ldots,N_{K_{n},n})$ the sizes of the
partition subsets $\{A_1, \ldots,A_{K_n}\}$. 
\citet[Proposition 9]{Pit(95)} shows that the above construction
leads to the joint distribution \eqref{epsm} of the random vector
$(K_{n},\mathbf{N}_{n})$. This is the reason with the above construction 
is referred to as the sequential or generative construction of the EP model.
Equation \eqref{sequential_k} clearly follows from the  
above construction since
\begin{displaymath}
\P(K_{n+1}=K_{n}+1\,|\,K_{n})=\frac{\alpha K_{n}+ \theta}{n+\theta}.
\end{displaymath}
Equation \eqref{DEFXIM1} also follows from the  
above construction as
\begin{align*}
&\P(\xi_{1,n+1}=1\,|\,\mathcal{F}_{n})=\frac{\alpha K_{n}+ \theta}{n+\theta},
\end{align*}
while for $n_1=1$,
\begin{displaymath}
\P(\xi_{1,n+1}=-1\,|\,\mathcal{F}_{n})=\frac{(1-\alpha)K_{1,n}}{n+\theta}.
\end{displaymath}
Moreover, we also deduce Equation \eqref{DEFXIMr} from the  
above construction since for all $r \geq 2$ and for $n_i=r-1$,
\begin{displaymath}
\P(\xi_{r,n+1}=1\,|\,\mathcal{F}_{n})=\frac{(r-1-\alpha)K_{r-1,n}}{n+\theta},
\end{displaymath}
whereas for $n_i=r$
\begin{displaymath}
\P(\xi_{r,n+1}=-1\,|\,\mathcal{F}_{n})=\frac{(r-\alpha)K_{r,n}}{n+\theta}.
\end{displaymath}
Finally, it shows that Equations \eqref{sequential_k}, \eqref{DEFXIM1} and \eqref{DEFXIMr} follow from the sequential or generative construction of the EP model.

\section*{Appendix B. Alternative proofs of the convergence in $\mathbb{L}^2$}
\renewcommand{\thesection}{\Alph{section}}
\renewcommand{\theequation}{\thesection.\arabic{equation}}
\setcounter{section}{2}
\setcounter{equation}{0}
\setcounter{thm}{0}
\subsection{$\mathbb{L}^2$ convergence of the number of partition sets}
We propose an alternative proof, without martingale, of the $\mathbb{L}^{2}$ convergence of $K_{n}$, 
\begin{equation}
\label{ASCVGLPAPP}
 \lim_{n \rightarrow+\infty} \dE\left[ \left( \frac{K_n}{n^{\alpha}} -S_{\alpha,\theta} \right)^2 \right]=0.
\end{equation}
The proof follows from the Poisson process construction of the EP model \citep[Proposition 9]{Pit(03)}. See also \citet[Chapter 4]{Pit(06)}. From \citet[Equation 3.11]{Pit(06)},
\begin{equation}\label{uncondK}
\P(K_{n}=k)=\frac{\left(\frac{\theta}{\alpha}\right)^{(k)}}{(\theta)^{(n)}}\mathscr{C}(n,k;\alpha).
\end{equation}
Denote $T_{\alpha,\theta}=S_{\alpha,\theta}^{-1/\alpha}$. By combining \citet[Equation 66]{Pit(03)} with \citet[Equation 2.61]{Cha(05)},
\begin{equation}\label{condK}
\P(K_{n}=k\,|\,T_{\alpha,\theta}=t)=V_{n,k}(t)\frac{\mathscr{C}(n,k;\alpha)}{\alpha^{k}},
\end{equation}
where
\begin{displaymath}
V_{n,k}(t)=\frac{1}{\Gamma(n-k\alpha)}\left(\frac{\alpha}{t^{\alpha}}\right)^{k}
\int_{0}^{1}v^{n-1-k\alpha}\frac{f_{\alpha}((1-v)t)}{f_{\alpha}(t)}\ddr v
\end{displaymath}
with $f_{\alpha}$ being the positive $\alpha$-Stable density function. Then, it follows from  \eqref{pitman_div} and \eqref{condK}, as well as the tower property of the conditional expectation, that
\begin{align}
\E[K_{n}S_{\alpha,\theta}]&=\int_0^\infty  \E[K_{n}S_{\alpha,\theta}| S_{\alpha,\theta}=s] f_{S_{\alpha,\theta}}(s)
\ddr s,\notag \\
&=\frac{\Gamma(\theta+1)}{\Gamma\left(\theta/\alpha+1\right)}\int_{0}^{+\infty}
\E[K_{n}| T_{\alpha,\theta}=t]t^{-\alpha-\theta}f_{\alpha}(t)\ddr t,\notag\\
&=\frac{\Gamma(\theta+1)}{\Gamma\left(\theta/\alpha+1\right)}\sum_{k=1}^{n}k\int_{0}^{+\infty}\P(K_{n}=k\,|\,T_{\alpha,\theta}=t) t^{-\alpha-\theta}f_{\alpha}(t)\ddr t,\notag\\
&=\frac{\Gamma(\theta+1)}{\Gamma\left(\theta/\alpha+1\right)}\sum_{k=1}^{n}\frac{k\mathscr{C}(n,k;\alpha)}{\Gamma(n-k\alpha)}I_{n,k}(t),
\label{parz_alt1}
\end{align}
where
\begin{align*}
I_{n,k}(t)&=\int_{0}^{+\infty}t^{-k\alpha-\alpha-\theta}\int_{0}^{1}v^{n-1-k\alpha}f_{\alpha}((1-v)t)\ddr v\ddr t, \\
&= \int_{0}^{+\infty}z^{-k\alpha-\alpha-\theta}\int_{0}^{1}v^{n-1-k\alpha}(1-v)^{k\alpha+\alpha+\theta-1}f_{\alpha}(z)\ddr v\ddr z.
\end{align*}
Since
\begin{equation*}
\int_{0}^{+\infty}z^{-k\alpha-\alpha-\theta}f_{\alpha}(z)\ddr z=\frac{\Gamma\left(\theta/\alpha +k+2\right)}{\Gamma(\alpha+\theta +k\alpha+1)},
\end{equation*}
and
\begin{equation*}
\int_{0}^{1}v^{n-k\alpha-1}(1-v)^{k\alpha+\alpha+\theta-1}\ddr v=\frac{\Gamma(n-k\alpha)\Gamma(k\alpha+\alpha+\theta)}{\Gamma(n+\alpha+\theta)},
\end{equation*}
the expectation \eqref{parz_alt1} reduces to
\begin{align}
\E[K_{n}S_{\alpha,\theta}]&=\frac{\Gamma(\theta+1)}{\Gamma\left(\theta/\alpha+1\right)\Gamma(n+\alpha+\theta)}\sum_{k=1}^{n} 
\frac{k\mathscr{C}(n,k;\alpha)}{(\alpha+\theta +k\alpha)}
\Gamma\left(\theta/\alpha +k+2\right), \notag\\
&=\frac{\Gamma(\theta+n)}{\Gamma(n+\alpha+\theta)}\sum_{k=1}^{n}k\left(\frac{\theta}{\alpha}+k\right)\frac{\left(\frac{\theta}{\alpha}\right)^{(k)}}{(\theta)^{(n)}}\mathscr{C}(n,k;\alpha).
\label{parz_alt4}
\end{align}
Hence, it follows from \eqref{uncondK} and \eqref{parz_alt4} that
\begin{equation}\label{parz_alt5}
\E[K_{n}S_{\alpha,\theta}]=\frac{\Gamma(\theta+n)}{\Gamma(n+\alpha+\theta)}\left(\frac{\theta}{\alpha}\E[K_{n}]+\E[K^{2}_{n}]\right).
\end{equation}
We are now in the position to prove \eqref{ASCVGLPAPP}. We get from \eqref{CVGMEANKNP} 
with $p=1$ and $p=2$ that
\begin{equation}
\label{parz_alt6}
\lim_{n \rightarrow+\infty} \frac{1}{n^{\alpha }}\dE[K_{n}]=\E[S_{\alpha,\theta}]
\hspace{0.8cm} \text{and} \hspace{0.8cm} 
\lim_{n \rightarrow+\infty} \frac{1}{n^{2\alpha}}\dE[K_{n}^2]=\E[S_{\alpha,\theta}^{2}],
\end{equation}
leading to
\begin{equation}
\label{parz_alt7}
\lim_{n\rightarrow+\infty}
\frac{1}{n^{\alpha }}\dE[K_{n}S_{\alpha,\theta}]=\E[S_{\alpha,\theta}^{2}].
\end{equation}
Therefore, we deduce from \eqref{parz_alt6} and \eqref{parz_alt7} that
\begin{align*}
\lim_{n\rightarrow+\infty}\E\left[\left(\frac{K_{n}}{n^{\alpha}}-S_{\alpha,\theta}\right)^{2}\right]
 &=\lim_{n\rightarrow+\infty}
\frac{1}{n^{2\alpha}}\dE[K_{n}^2] -\frac{2}{n^\alpha}\dE[K_{n}S_{\alpha,\theta}]+\E[S_{\alpha,\theta}^{2}], \\
&=\E[S_{\alpha,\theta}^{2}]-2\E[S_{\alpha,\theta}^{2}]+\E[S_{\alpha,\theta}^{2}]=0.
\end{align*}
Similar arguments can be applied in order to show the convergence in $\mathbb{L}^{p}$ of $n^{-\alpha}K_{n}$ to $S_{\alpha,\theta}$, for any integer $p\geq1$, thus providing an alternative proof of \eqref{ASCVGLP}.

\subsection{$\mathbb{L}^2$ convergence of the number of partition subsets of size $r$}
We still rely on the Poisson process construction of the EP model \citep[Proposition 9]{Pit(03)} to propose an alternative proof of the $\mathbb{L}^{2}$ convergence of $K_{r,n}$, 
\begin{equation}
\label{ASCVGLPAPPBLOCK}
 \lim_{n \rightarrow+\infty} \dE\left[ \left( \frac{K_{r,n}}{n^{\alpha}} -p_{\alpha}(r)S_{\alpha,\theta} \right)^2 \right]=0.
\end{equation}
Denote $T_{\alpha,\theta}=S_{\alpha,\theta}^{-1/\alpha}$. By combining \citet[Equation 66]{Pit(03)} with \citet[Proposition 1]{Fav(13)},
\begin{equation}\label{condKr}
\E[K_{r,n}\,|\,T_{\alpha,\theta}=t]=p_{\alpha}(r) (n)_{(r)}\sum_{k=1}^{n}V_{n,k}(t)\frac{\mathscr{C}(n-r,k-1;\alpha)}{\alpha^{k}}
\end{equation}
where
\begin{displaymath}
V_{n,k}(t)=\frac{1}{\Gamma(n-k\alpha)}\left(\frac{\alpha}{t^{\alpha}}\right)^{k}
\int_{0}^{1}v^{n-1-k\alpha}\frac{f_{\alpha}((1-v)t)}{f_{\alpha}(t)}\ddr v
\end{displaymath}
with $f_{\alpha}$ being the positive $\alpha$-Stable density function. Consequently, we can write that
$$
\E[K_{r,n}S_{\alpha,\theta}]=\frac{\Gamma(\theta+1)}{\Gamma\left(\theta/\alpha+1\right)}\int_{0}^{+\infty}
\E[K_{r,n}| T_{\alpha,\theta}=t]t^{-\alpha-\theta}f_{\alpha}(t)\ddr t,
$$
thus obtaining, along lines similar to the computation of $\E[K_{n}S_{\alpha,\theta}]$ in
\eqref{parz_alt4}, that 
\begin{align*}
\E[K_{r,n}S_{\alpha,\theta}]
&=\frac{\Gamma(\theta+1)}{\alpha\Gamma(n+\alpha+\theta)}p_{\alpha}(r)(n)_{(r)}\sum_{k=1}^{n}\left(\frac{\theta}{\alpha}+1\right)^{(k)}\mathscr{C}(n-r,k-1;\alpha)\\
&=\frac{\Gamma(\theta+1)}{\alpha\Gamma(n+\alpha+\theta)}p_{\alpha}(r)(n)_{(r)}\frac{\Gamma\left(\theta/\alpha+2\right)}{\Gamma\left(\theta/\alpha+1\right)}\sum_{k=0}^{n-r}\left(\frac{\theta}{\alpha}+2\right)^{(k)}\mathscr{C}(n-r,k;\alpha), 
\end{align*}
which leads to
\begin{equation}\label{parz_alt1_kr}
\E[K_{r,n}S_{\alpha,\theta}]=\frac{\Gamma(\theta+1)}{\Gamma(n+\alpha+\theta)}
p_{\alpha}(r) (n)_{(r)}\left(\frac{\alpha +\theta}{\alpha^2}\right)(\theta+2\alpha)^{(n-r)}.
\end{equation}
We are now in position to prove \eqref{ASCVGLPAPPBLOCK}. We find from \eqref{LP_M3} 
with $p=2$ that
\begin{equation}
\label{LPAN2}
\lim_{n \rightarrow+\infty} \frac{1}{n^{2\alpha}}\dE[K_{r,n}^2]=\big(p_r(\alpha)\big)^2
\E[ \big(S_{\alpha,\theta}\big)^{2}].
\end{equation}
In addition, we deduce from \eqref{parz_alt1_kr} that
\begin{equation}\label{parz_alt3_kr}
\lim_{n\rightarrow+\infty}
\frac{1}{n^{\alpha}}
\E[K_{r,n}S_{\alpha,\theta}]=p_r(\alpha)\left(\frac{\alpha+ \theta}{\alpha^2}\right)\frac{\Gamma(\theta+1)}{\Gamma(2\alpha+\theta)}=p_r(\alpha)\E[S_{\alpha,\theta}^2].
\end{equation}
Therefore, it follows from \eqref{LPAN2} and \eqref{parz_alt3_kr} that
\begin{align*}
&\lim_{n\rightarrow+\infty}\E\left[\left(\frac{K_{r,n}}{n^{\alpha}}-p_r(\alpha)S_{\alpha,\theta}\right)^{2}\right]\\
&
\quad=\lim_{n\rightarrow+\infty}
\frac{1}{n^{2\alpha}}\dE[K_{r,n}^2] -\frac{2p_r(\alpha)}{n^\alpha}\dE[K_{r,n}S_{\alpha,\theta}]+\big(p_r(\alpha)\big)^2
\E[ \big(S_{\alpha,\theta}\big)^{2}], \\
& \\
&\quad =\big(p_r(\alpha)\big)^2\E[S_{\alpha,\theta}^{2}]-2\big(p_r(\alpha)\big)^2\E[S_{\alpha,\theta}^{2}]+\big(p_r(\alpha)\big)^2
\E[ \big(S_{\alpha,\theta}\big)^{2}]=0.
\end{align*}
Similar arguments can be applied to show the convergence in $\mathbb{L}^{p}$ of $n^{-\alpha}K_{r,n}$ to $p_{\alpha}(r)S_{\alpha,\theta}$, for any integer $p\geq1$, thus leading to an alternative proof of \eqref{ASCVGLP_M}.


\section*{Acknowledgement}

The authors would like to thank two anonymous referees for their constructive comments which helped to improve the paper substantially. Stefano Favaro received funding from the European Research Council (ERC) under the European Union's Horizon 2020 research and innovation programme under grant agreement No 817257. Stefano Favaro acknowledge the financial support from the Italian Ministry of Education, University and Research (MIUR), ``Dipartimenti di Eccellenza" grant 2023-2027.



\begin{thebibliography}{9}

\bibitem[Arratia et al.(1992)]{Arr(92)}
\textsc{Arratia, R., Barbour, A.D. and Tavar\'e, S.} (1992). Poisson process approximations for the Ewens sampling formula. \textit{The Annals of Applied Probability} \textbf{2}, 519--535

\bibitem[Arratia et al.(2003)]{Arr(03)}
\textsc{Arratia, R., Barbour, A.D. and Tavar\'e, S.} (2003). \textit{Logarithmic combinatorial structures: a probabilistic approach}. EMS Monographs in Mathematics.

\bibitem[Balocchi et al.(2023)]{Bal(23)}
\textsc{Balocchi, C., Favaro, S. and Naulet, Z.} (2023). Bayesian nonparametric inference for ``species-sampling" problems. \textit{Preprint arXiv:2203.06076}.

\bibitem[Charalambides(2005)]{Cha(05)}
\textsc{Charalambides} (2005) \textit{Combinatorial methods in discrete distributions.} Wiley.

\bibitem[Charalambides(2007)]{Cha(07)}
\textsc{Charalambides, C.A.} (2007). Distributions of random partitions and their applications. \textit{Methodology and Computing in Applied Probability} \textbf{9}, 163--193.

\bibitem[Dolera and Favaro(2020)]{Dol(20)}
\textsc{Dolera, E. and Favaro, S.} (2020). A Berry--Esseen theorem for Pitman's $\alpha$--diversity. \textit{The Annals of Applied Probability} \textbf{30}, 847--869.

\bibitem[Dolera and Favaro(2021)]{Dol(21)}
\textsc{Dolera, E. and Favaro, S.} (2021). A compound Poisson perspective of Ewens-Pitman sampling model. \textit{Mathematics} \textbf{9}, 2820.

\bibitem[Duflo(1997)]{Duf(97)}
\textsc{Duflo, M.} (1997) \textit{Random iterative models}. Springer-Verlag.

\bibitem[Ewens(1972)]{Ewe(72)}
\textsc{Ewens, W.} (1972). The sampling theory or selectively neutral alleles. \textit{Theoretical Population Biology} \textbf{3}, 87--112.

\bibitem[Favaro et al.(2015)]{Fav(15)}
\textsc{Favaro, S. and Feng, S.} (2015). Large deviation principles for the Ewens-Pitman sampling model. \textit{Electronic Journal of Probability} \textbf{20}, 1-26

\bibitem[Favaro et al.(2018)]{Fav(18)}
\textsc{Favaro, S., Feng, S. and Gao. F.} (2018). Moderate deviations for Ewens-Pitman sampling models. \textit{Sankhya A} \textbf{80}, 330--341.

\bibitem[Favaro et al.(2013)]{Fav(13)}  
\textsc{Favaro, S., Lijoi, A. and Pr\"unster, I.} (2013). Conditional formulae for Gibbs-type exchangeable random partitions. \textit{The Annals of Applied Probability} \textbf{23}, 1721--1754.

\bibitem[Favaro and Naulet(2023)]{Fav(23)}  
\textsc{Favaro, S. and Naulet, Z.} (2023). Near-optimal estimation of the unseen under regularly varying tail populations. \textit{Bernoulli} \textbf{29},  3423--3442.

\bibitem[Feng and Hoppe(1998)]{Fen(98)}  
\textsc{Feng, S. and Hoppe, F.M.} (1998). Large deviation principles for some random combinatorial structures in population genetics and Brownian motion. \textit{The Annals of Applied Probability}  \textbf{8}, 975--994.

\bibitem[Hall and Heyde(1980)]{Hal(80)}
\textsc{Hall, M. and Heyde, C.C.} (1980) \textit{Martingale limit theory and its applications.} Academic Press.

\bibitem[Heyde(1977)]{Hey(77)}
\textsc{Heyde, C.C.} (1977). On central limit and iterated logarithm supplements to the martingale convergence theorem. \textit{Journal of Applied Probability} \textbf{14}, 758--775.

\bibitem[Korwar and Hollander(1973)]{Kor(73)}
\textsc{Korwar, R.M. and Hollander, M.} (1973). Contribution to the theory of Dirichlet processes. \textit{The Annals of Probability} \textbf{1}, 705--711.

\bibitem[Oliveira et al.(2022)]{Oli(22)}
\textsc{Oliveira, R.I., Pereira, A. and Ribeiro, R.} (2022). Concentration in the generalized Chinese restaurant process. \textit{Sankhya A} \textbf{80}, 628-670.

\bibitem[Perman et al.(1992)]{Per(92)}
\textsc{Perman, M., Pitman, J. and Yor, M.} (1992). Size-biased sampling of Poisson point processes and excursions. \textit{Probab. Theory Related Fields} \textbf{92}, 21--39.

\bibitem[Pitman(1995)]{Pit(95)}
\textsc{Pitman, J.} (1995). Exchangeable and partially exchangeable random partitions. \textit{Probability Theory and Related Fields} \textbf{102}, 145--158.

\bibitem[Pitman(2003)]{Pit(03)}
\textsc{Pitman, J.} (2003). Poisson-Kingman partitions. In \textit{Science and Statistics: A Festschrift for Terry Speed}, Goldstein, D.R. Eds. Institute of Mathematical Statistics.

\bibitem[Pitman and Yor(1997)]{Pit(97)}
\textsc{Pitman, J. and Yor, M.} (1997). The two parameter Poisson-Dirichlet distribution derived from a stable subordinator. \textit{The Annals of Probability} \textbf{25}, 855--900.

\bibitem[Pitman(2006)]{Pit(06)}
\textsc{Pitman, J.} (2006). \textit{Combinatorial stochastic processes}. Lecture Notes in Mathematics, Springer Verlag.

\bibitem[Zabell(2005)]{Zab(05)}
\textsc{Zabell, S.L.}  (2005). \textit{Symmetry and its discontents: essays on the history of inductive probability}. Cambridge University Press.

\bibitem[Zolotarev(1986)]{Zol(86)}
\textsc{Zolotarev, V.M.}  (1986). \textit{One-dimensional Stable Distributions}. American Mathematical Society.

\end{thebibliography}
\end{document}